\documentclass{amsart}

\usepackage{hyperref}

\usepackage{amsmath, amsthm,  amsfonts, amssymb}
\usepackage{mathrsfs} 
\usepackage{mathscinet} 
\usepackage[thinlines]{easytable}
\usepackage{pifont}
\usepackage{scalerel}
\usepackage{epigraph}
\usepackage{array}
\usepackage[shortlabels]{enumitem} 
\usepackage{xfrac}
\usepackage{units}
\usepackage{soul}
\usepackage{xcolor}
\usepackage{accents}
\usepackage{tikz-cd}
\usepackage{tikz}
\usepackage{caption}
\usetikzlibrary{automata,positioning,arrows}

 \newcommand{\R}{\mathbb{R}}
 \newcommand{\N}{\mathbb{N}}
 \newcommand{\Z}{\mathbb{Z}}
 \newcommand{\C}{\mathbb{C}}
 
 \newcommand{\su}{SU(n)/T}

\newcommand{\basicbg}{\left(P_b, Y, \pi_b  \right)}
\newcommand{\cupbg}{\left(P_c, X\right)}
\newcommand{\stimes}{{\times}}
\newcommand{\sstimes}{{   \, \stimes \,}}
\newcommand{\pullbackbasicbg}{ p^{-1}\left(P_b, Y  \right)}
\newcommand{\Pbasic}{  P_{b}}
\newcommand{\PbasicT}{  P_{b, T}}
\newcommand{\basicbgtorus}{\left( P_{b, T}, Y_T  \right)}

\DeclareMathOperator*{\Motimes}{\text{\raisebox{0.25ex}{\scalebox{0.8}{$\bigotimes$}}}}
\DeclareMathOperator*{\Moplus}{\text{\raisebox{0.25ex}{\scalebox{0.8}{$\bigoplus$}}}}

\newlength{\dhatheight}

\newlength{\mylen}
\setbox1=\hbox{$\bullet$}\setbox2=\hbox{\tiny$\bullet$}
\DeclareMathOperator{\DD}{DD}

\DeclareMathOperator{\tr}{tr}

\DeclareMathOperator{\hol}{hol}

\newcommand{\CC}{{\mathbb C}}
\newcommand{\RR}{{\mathbb R}}
\newcommand{\ZZ}{{\mathbb Z}}

\let\oldtocsection=\tocsection

\let\oldtocsubsection=\tocsubsection

\let\oldtocsubsubsection=\tocsubsubsection

\renewcommand{\tocsection}[2]{\hspace{0em}\oldtocsection{#1}{#2}}
\renewcommand{\tocsubsection}[2]{\hspace{1em}\oldtocsubsection{#1}{#2}}
\renewcommand{\tocsubsubsection}[2]{\hspace{2em}\oldtocsubsubsection{#1}{#2}}
\begin{document}
\def\hmath#1{\text{\scalebox{1.6}{$#1$}}}
\def\lmath#1{\text{\scalebox{1.4}{$#1$}}}
\def\mmath#1{\text{\scalebox{1.2}{$#1$}}}
\def\smath#1{\text{\scalebox{.8}{$#1$}}}

\def\hfrac#1#2{\hmath{\frac{#1}{#2}}}
\def\lfrac#1#2{\lmath{\frac{#1}{#2}}}
\def\mfrac#1#2{\mmath{\frac{#1}{#2}}}
\def\sfrac#1#2{\smath{\frac{#1}{#2}}}
\def\pow{^\mmath}

%

\theoremstyle{plain}
\newtheorem{theorem}{Theorem}[section]
\newtheorem{corollary}[theorem]{Corollary}
\newtheorem{lemma}[theorem]{Lemma}
\newtheorem{proposition}[theorem]{Proposition}
\newtheorem{question}[theorem]{Question}

\theoremstyle{definition}
\newtheorem{definition}[theorem]{Definition}

\theoremstyle{remark}
\newtheorem{example}{Example}[section]
\newtheorem{note}{Note}[section]
\newtheorem{exercise}{Exercise}[section]
\newtheorem{remark}{Remark}[section]
\newtheorem{fact}{Fact}[section]

\numberwithin{equation}{section}
\numberwithin{figure}{section}


\title[Weyl map and bundle gerbes]{The Weyl map and bundle gerbes}

 \author[K. E. Becker]{Kimberly E. Becker}
  \address[K. E. Becker]
  {Wolfson College\\
  Linton Road\\
  Oxford, OX2 6UD \\
United Kingdom}
  \email{kimberly.becker@univ.ox.ac.uk}

 \author[M. K. Murray]{Michael K. Murray}
  \address[M. K. Murray]
  {School of Mathematical Sciences\\
  University of Adelaide\\
  Adelaide, SA 5005 \\
  Australia}
  \email{michael.murray@adelaide.edu.au}
  
  \author[D. Stevenson]{Daniel Stevenson}
  \address[D. Stevenson]
  {School of Mathematical Sciences\\
  University of Adelaide\\
  Adelaide, SA 5005 \\
  Australia}
  \email{daniel.stevenson@adelaide.edu.au}

\date{\today}

\thanks{The second two authors acknowledge support under the  Australian
Research Council's {\sl Discovery Projects} funding scheme (project number DP180100383).
The first author acknowledges the support of an {\sl Australian Government Research Training Program Scholarship}.\\
\copyright 2019. This manuscript version is made available under the CC-BY-NC-ND 4.0 license \url{http://creativecommons.org/licenses/by-nc-nd/4.0/}}



\begin{abstract}
We introduce the notion of a general
cup product bundle gerbe and use it to define the Weyl bundle gerbe on $T \times \su$.  
The Weyl map from $T \times \su$ to $SU(n)$ is then used to 
show that the pullback of the basic bundle gerbe on $SU(n)$ defined
by the second two authors is stably isomorphic to the Weyl bundle gerbe as $SU(n)$-equivariant bundle gerbes. 
Both bundle gerbes come equipped with connections and curvings
and by considering the holonomy of these we show that these bundle 
gerbes are not $D$-stably isomorphic.

\end{abstract}

\maketitle

\tableofcontents
%
%

\section{Introduction} 
The \textit{basic bundle gerbe} is defined  over a compact, simple, simply connected Lie group $G$ and has Dixmier-Douady class equal to a generator of $H^3(G, \Z)\simeq \Z$. 
In this work, we will restrict our study to the case where $G=SU(n)$ and use the finite-dimensional 
construction of the
 basic bundle gerbe in this case given by  the second two authors \cite{53UNITARY}, see also \cite{mickelsson03}. The history of the 
various constructions of the basic bundle gerbe and its extensions to other groups can be found in the Introduction  to \cite{53UNITARY}. 
 As well as the construction of the basic bundle gerbe over $SU(n)$, \cite{53UNITARY} gives an explicit connection and curving {on this bundle gerbe} with three-curvature equal to $2 \pi i$ times  the basic  $3$-form
$$
\nu = -\frac{1}{24 \pi^2} \tr(g^{-1} dg)^3 
$$
 on $SU(n)$. These explicit formulae will be used extensively in our work.
 
 Our work here can be understood as a follow up to \cite{53UNITARY}. Namely, we study the pullback of 
the basic gerbe over $SU(n)$ from \cite{53UNITARY}  by the Weyl map 
\begin{align*}
 p\colon T \times \su &\xrightarrow{} SU(n) \\
 p(t, hT) &= h t h^{-1}
 \end{align*}
and explain why the pulled back bundle gerbe decomposes into simpler objects.   Our motivation for this is the following observation.
The pullback of the basic $3$-form by the Weyl map, $p^*\nu$, {defines a class in} $H^3\left(T\times \su\right)$. By the Kunneth formula,
 and noting that the cohomology of $\su$ vanishes in odd degree \cite{odd}, we see that 
$$
[p^*\nu] \in H^3(T)\oplus \left(H^2\left(\su\right)\otimes H^1(T)\right).
$$ 
It follows from $T$ being abelian that the restriction of $\nu$ to $T$ vanishes. Therefore 
$$
[p^*\nu] \in H^2\left(\su\right)\otimes H^1(T).
$$ 
For this reason, we expect that the three-curvature of the pullback of the basic bundle gerbe by the 
Weyl map to equal a sum of wedge products of $1$-forms and $2$-forms (modulo exact forms).
We are interested in the impact this fact has on the geometry of the pulled back basic bundle gerbe. 
It follows from work of Johnson in \cite{stuartthesis} that a bundle gerbe whose Dixmier-Douady class is the cup product of a one-class
 and a two-class can be realised by a geometric cup product construction from a $U(1)$-valued function whose winding class is the one-class
  and a line bundle whose chern class is the two-class. We expect therefore that the pullback of the basic bundle gerbe by the Weyl map is 
  stably isomorphic to a particular product of
these cup product bundle gerbes which we call the Weyl bundle gerbe.

Following Johnson \cite{stuartthesis} we introduce the \textit{cup product bundle gerbe} construction. The \textit{Weyl bundle gerbe} is then
  defined as a reduced product (see Section~\ref{subsec:bundle gerbes} below) 
  of certain cup product bundle gerbes over 
  $T\times \su$ for $T$ the maximal torus of 
  $SU(n)$ consisting of diagonal
   matrices. We call any such bundle gerbe that 
   is the reduced product of cup product bundle 
   gerbes a \textit{general cup product bundle gerbe}. 
   In our main theorem (Theorem \ref{finalresult}), we 
\begin{enumerate}[(1)]
	\item describe an explicit $SU(n)$-equivariant stable isomorphism of the Weyl bundle gerbe and the pullback of the basic bundle gerbe over 
	$SU(n)$ via the Weyl map $p:T\times \su\to SU(n);$ 
	\item explicitly define a $2$-form $\beta\in \Omega^2\left(T\times \su \right)$ that satisfies $d\beta = \omega_{p^*b} - \omega_c$ for $\omega_{p^*b}$
	 and $\omega_c$ the three-curvatures associated to the pullback connective data on the pullback of the basic bundle gerbe and natural connective data
	  on the Weyl bundle gerbe, respectively$;$ 
	\item show there is no general cup product bundle gerbe that is stably isomorphic to the pullback of the basic bundle gerbe such that the 
	three-curvature $\omega_c'$ associated to the induced connective data on the general cup product bundle gerbe satisfies $\omega_c' = \omega_{p^*b}$; {and}
	\item consider the holonomies of our bundle gerbes to show that the pullback of the basic bundle gerbe and the Weyl bundle gerbe are not $D$-stably 
	isomorphic with respect to their natural connective data. 
\end{enumerate}

There is a long history of exploiting  the attractive features of the  Weyl map going back to Weyl's proof of the Integral Formula \cite{weyl1} and 
the  $K$-theory of compact Lie groups \cite{atiyahusingweylmap}. Cup product constructions 
similar to ours have been used by Brylinski \cite{brylinski}  
to construct projective unitary group bundles, in index theory  \cite{MatMelSin}
and  in twisted $K$-theory \cite{HarMic, BarHek}. We note also that \cite[Section 8.2]{MatMel} gives a general construction of the cup product bundle gerbe for a 
decomposable Dixmier-Douady class. 

In summary, in Section \ref{sec:bundlegerbes} we briefly review basic results, definitions and notation from the theory of bundle 
gerbes. In Section \ref{section cup product bundle gerbes}, we introduce the notions of cup product and general cup product bundle gerbes, and study the geometry of the former. Criteria for general cup product bundle gerbes to be stably isomorphic are also considered. Next, in Section  \ref{sec:weyl-bundle-gerbe}, we apply {the} theory {from Section} \ref{section cup product bundle gerbes} to construct
the Weyl bundle gerbe over $T \times \su$. The pullback of the basic bundle gerbe is 
considered in Section \ref{sec:basic-pullback}, {and} we show that it 
is also a general cup product bundle gerbe over $T \times \su$.  The stable 
isomorphism between the pullback of the basic bundle gerbe and the Weyl bundle gerbe is given in Section \ref{sec:stable}, where we exploit the results of 
Section \ref{section cup product bundle gerbes} using {the fact} that both bundle gerbes are general cup product bundle gerbes. We conclude Section \ref{sec:stable} by demonstrating that these bundle gerbes, with their given connections and 
curvings, are not stably isomorphic, i.e. they are not $D$-stably isomorphic. Our {results} are summarised in Theorem \ref{finalresult}.

%
%

\section{Bundle gerbes}
\label{sec:bundlegerbes}

We review some notation and basic facts about bundle gerbes.
For more detail on bundle gerbes see \cite{1996bundlegerbes, INTROTOBG, Bec}.

\subsection{Surjective submersions}
\label{subsec:surj submer}
Let $\pi \colon Y \to M$ be a surjective submersion.  
Denote by $Y^{[p]}$ the $p$-fold fibre product of $Y$; 
note that the canonical map $Y^{[p]}\to M$ 
is also a surjective submersion.  Define $\pi_i \colon Y^{[p+1]} \to Y^{[p]}$ by omitting the $i$-th
entry for each $i = 1, \dots, p+1$.  Notice that this means that the two maps
$Y^{[2]} \to Y$ are (perhaps confusingly) $\pi_1((y_1, y_2)) = y_2$ and $\pi_2((y_1, y_2)) = y_1$.
If $g \colon Y^{[p]} \to A$ is a map to an abelian group $A$, define 
$\delta(g) \colon Y^{[p+1]} \to A$ by $\delta(g) = g \circ \pi_1 - g \circ \pi_2 + \cdots $ and 
 if $g \colon M \to A$ define $\delta(g) \colon Y \to A$ by $\delta(g) = g \circ \pi$. 
 
 Similarly,
 if $\omega \in \Omega^q(Y^{[p]}) $ is a $q$-form, define $\delta(\omega) \in \Omega^q(Y^{[p+1]})$
by $\delta(\omega) = \pi_1^*(\omega) - \pi^*_2(\omega) + \cdots$ and likewise $\delta(\omega) = 
 \pi^*(\omega)$ if $\omega \in \Omega^q(M)$. It is straightforward to check that the {\em fundamental complex}
 defined by 
 \begin{equation}
 \label{eq:fund-comp} 
 0 \to \Omega^q(M) \stackrel{\delta}{\to} \Omega^q(Y) \stackrel{\delta}{\to} \Omega^q(Y^{[2]}) 
 \stackrel{\delta}{\to} \Omega^q(Y^{[3]}) \stackrel{\delta}{\to}\cdots 
\end{equation}
is exact \cite[Section 8]{1996bundlegerbes}\footnote{The proof given there is actually for $Y \to M$ a fibration 
but can be adapted to the case where $Y\to M$ 
is a surjective submersion using the fact that surjective submersions admit local sections.}. 

If $K \to Y^{[p]}$ is a Hermitian line bundle, define a Hermitian 
line bundle $\delta(K)$ over $Y^{[p+1]}$ by 
$ \delta(K) = \pi_1^{-1}(K) \otimes \pi^{-1}_2(K)^* \otimes \cdots$.
Note that the Hermitian line bundle $\delta\delta(K)$ over 
$Y^{[p+2]}$ has a canonical trivialisation.  If $K$ is equipped with a connection 
$\nabla_K$, then there is an induced connection $\delta(\nabla_K)$ on 
$\delta(K)$. 
 
\subsection{Bundle gerbes} 
\label{subsec:bundle gerbes}
Let $M$ be a manifold.  Recall that a {\em bundle gerbe} over $M$, 
denoted $(P, Y)$ or $(P, Y, \pi)$, consists of a surjective 
submersion $\pi: Y \to M$ and a Hermitian line bundle $P \to Y^{[2]}$. {The bundle $P$} is  equipped with a 
 {\em bundle gerbe multiplication} $\pi_3^{-1}(P) \otimes \pi_1^{-1}(P) \to \pi^{-1}_2(P)$
 which is associative in the sense that the two different ways of mapping 
 $$
 P_{(y_1, y_2)} \otimes P_{(y_2, y_3)} \otimes P_{(y_3, y_4)} \to P_{(y_1, y_4)}
 $$
 agree for any $(y_1, y_2, y_3, y_4) \in Y^{[4]}$. 
 
 If $(P, Y)$ is a bundle gerbe over $M$ and $(Q, X)$ is a bundle gerbe over $N$, then a 
 morphism of bundle gerbes $(P,Y)\to (Q,X)$ 
 is a triple of maps $f \colon M \to N$, $g \colon Y \to X$ and $h \colon P \to Q$.
 These have to satisfy:  $g$ covers $f$ and thus induces a map $g^{[2]} \colon Y^{[2]} \to X^{[2]}$ and 
 $h \colon P \to Q$ is a bundle morphism covering $g^{[2]}$.

A bundle gerbe is called {\em trivial} if there is a Hermitian line bundle $R \to Y$ and an isomorphism 
of $P$ to $\delta(R) = \pi_1^{-1}(R) \otimes \pi_2^{-1}(R)^* $ such that  the bundle gerbe 
product on $P$ commutes with the obvious contraction 
\begin{equation}
\label{eq:triv}
R_{y_2} \otimes R^*_{y_1} \otimes R_{y_3} \otimes R^*_{y_2} \to R_{y_3} \otimes R^*_{y_1}.
\end{equation}
A {\em trivialisation} of $(P, Y)$ is a choice of such a {\em trivialising line bundle} $R$ and isomorphism \eqref{eq:triv}.

If $(P, Y)$ and $(Q, X)$ are bundle gerbes over $M$, we define the {\em dual} of $(P, Y)$ 
by $(P^*, Y)$ with the obvious multiplication and their 
{\em product} $(P, Y) \otimes (Q, X)$  by $(P \otimes Q, Y \times_M X)$, where
$Y \times_M X$ is the fibre product of $Y \to M$ and $X \to M$ and 
$$
(P\otimes Q)_{((y_1, x_1), (y_2, x_2))} = P_{(y_1, y_2)} \otimes Q_{(x_1, x_2)}
$$
with the obvious multiplication.  In the case that $X = Y$ we have the diagonal inclusion of $Y$ into 
$Y \times_M Y$ and we may use this to pull back $P \otimes Q$ to define 
the {\em reduced product} $(P\otimes_R Q, Y)$.

We say that $(P, Y)$ and $(Q, X)$ are {\em stably isomorphic} if there exists a trivialisation 
of $(P, Y)^* \otimes (Q, X)$. Similarly a {\em stable isomorphism} from $(P, Y)$
to $(Q, X)$ is a choice of such a trivialisation.  It is shown in \cite{dannybundle2gerbes} that
stable isomorphisms can be composed.  The notion of morphism introduced above of course leads to a notion 
of isomorphism, which is much stronger than the notion of stable isomorphism.  We use the notations $\cong$ and $\cong_{\text{stab}}$ 
for the notions of isomorphism and stable isomorphism respectively. 

Associated to a bundle gerbe $(P, Y)$ is a characteristic class 
called its {\em Dixmier-Douady class} or \textit{DD-class},
$\DD(P, Y) \in H^3(M, \ZZ)$,
which determines exactly the stable isomorphism class of the bundle gerbe \cite{stableiso}. 

If $(P, Y)$ is a bundle gerbe over $M$ and 
$f \colon N \to M$, then $f^{-1}(Y) \to N$ is a
surjective submersion and  we have 
$f^{[2]} \colon f^{-1}(Y)^{[2]} \to Y^{[2]}$. 
The pullback of $P$ by $f^{[2]}$, more conveniently 
called $f^{-1}(P)$, then inherits a 
natural bundle gerbe multiplication.  The bundle 
gerbe $(f^{-1}(P), f^{-1}(Y))$ over $N$
is called the {\em pullback} of $(P, Y)$ by $f$.  
{Pulling back} preserves products and duals 
and is natural for the Dixmier-Douady class.  For more details see \cite{Bec}.

Similarly, we can consider {the pullback of bundle gerbes over $M$ by} morphisms of surjective submersions {over $M$}.  {These are} maps $f \colon X \to Y$ of surjective
submersions $X \to M$ and $Y \to M$ covering the identity map $M \to M$.  If $(P, Y)$ is a bundle
gerbe we can pull back using $f^{[2]}$ to obtain a bundle gerbe we denote by $(f^{-1}(P), X)$ over $M$.
A basic fact \cite[Proposition 3.4]{stableiso} is that $(f^{-1}(P), X)$ is stably isomorphic to $(P, Y)$.
Moreover, there is a canonical choice of stable isomorphism.   This implies immediately that the product
and reduced product of two bundle gerbes are canonically stably isomorphic.

Finally,  if $G$ is a Lie group we can consider (strongly) equivariant bundle gerbes 
where $G$ acts on $Y \to M$ and $P$, {preserving all relevant structure},  {for a precise definition} see, for example, \cite{MurRobSte}. The notion of stable isomorphism extends to this case by requiring
that $R$ also have a $G$-action and that the isomorphism $\delta(R) \cong P$ be $G$-equivariant. 
We use the obvious notation $ \cong_{G}$ and $\cong_{G\text{-stab}}$ for $G$-equivariant isomorphisms and 
stable isomorphism between $G$-equivariant bundle gerbes. 
\subsection{Connections, curvings {and holonomy}} Any bundle gerbe $(P, Y, \pi)$ admits a {\em bundle gerbe connection} $\nabla$ which is a Hermitian connection on $P$ 
preserving the bundle gerbe multiplication.  As a result of this condition, the curvature $F_\nabla \in \Omega^2(Y^{[2]})$
of a bundle gerbe connection satisfies $\delta(F_\nabla) = 0$.   It follows from exactness of the 
fundamental complex \eqref{eq:fund-comp} that there exist imaginary $2$-forms $f \in \Omega^2(Y)$ satisfying $\delta(f) = F_\nabla$. 
A choice of such an $f\in \Omega^2(Y)$ is called a {\em curving} for $\nabla$, and the pair $(\nabla, f)$ is called 
\textit{connective data} for $(P, Y)$. Notice that the curving is determined only up to addition of the pullback of a form in $\Omega^2(M)$.  Commutativity of the maps $\delta$ in the 
fundamental complex with exterior differentiation implies $\delta(d f) = d \delta(f) = d F_\nabla = 0$,
so that $df = \pi^*(\omega)$ for a unique $\omega \in \Omega^3(M)$ called the {\em three-curvature} of $(\nabla, f)$. 
We have that $d \omega = 0$ and the class of $\omega/{2 \pi i }$ in de Rham cohomology is the image
of $ DD(P, Y)$ in real cohomology under the de Rham isomorphism.

Connective data on bundle gerbes naturally induce connective data on dual, pullback, reduced product and product bundle gerbes. 
For more on this see \cite{Bec}. 

 There is a notion of stable isomorphisms of bundle gerbes with connective data. {Following Johnson}  \cite{stuartthesis} we say two bundle gerbes $(P, Y)$ and $(Q, X)$ with 
 connective data are \textit{D-stably isomorphic}, denoted $(P, Y)\cong_{D\text{-stab}} (Q, X)$, if $(P, Y)\cong_{\text{stab}} (Q, X)$ 
 and this stable isomorphism preserves connections and curvings. Here the trivial bundle gerbe $P^*\otimes Q$ is assumed to have 
 \textit{trivial connective data}, i.e. connective data of the form $(\delta(\nabla_R), F_{\nabla_R})$ for $\nabla_R$ a connection 
 on a line bundle $R\to Y$ with curvature $F_{\nabla_R}$. The $D$-stable isomorphism classes of bundle gerbes over $M$ (or 
 \textit{Deligne classes}) are in bijective correspondence with the Deligne cohomology group $H^3(M, \Z(3)_D)$ \cite[Theorem 4.1]{stableiso}. For more on this, see \cite{stuartthesis,stableiso}.
 
If $(P, Y)$ is a bundle gerbe over an oriented two-dimensional manifold $\Sigma$, then $H^3(\Sigma, \ZZ) = 0$ and there exists a Hermitian line bundle 
$R \to Y$ trivialising $P$.  Suppose $\nabla$ is a bundle gerbe connection on $P$ 
and $\nabla_R$ is a connection on $R$. Denote by $\delta(\nabla_R)$ the connection 
on $(P, Y)$ induced by the isomorphism $P \cong \delta(R)$.  As both $\nabla$ and $\delta(\nabla_R)$
are bundle gerbe connections, it follows that $\nabla = \delta(\nabla_R) + \alpha$
for $\alpha \in \Omega^1(Y^{[2]})$ satisfying $\delta(\alpha) = 0$.  From exactness 
of the fundamental complex we can solve $\alpha = \delta(\beta)$ for some $\beta \in \Omega^1(Y)$
and thus $\nabla = \delta(\nabla_R  +\beta)$. It follows that we may suppose without loss of generality 
that $\delta(\nabla_R) = \nabla$.  Denote by $F_{\nabla_R}$ the curvature
of $\nabla_R$ and note that $\delta(F_{\nabla_R}) = F_{\nabla}$.  
Then $\delta(f - F_{\nabla_R}) = F_\nabla - F_\nabla  = 0$ so we have 
$f - F_{\nabla_R} = \pi^*(\mu)$ for $\mu \in \Omega^2(\Sigma)$. 
We define the {\em holonomy} of $(\nabla, f)$ over $\Sigma$, $\hol(\nabla, f)$, by $\exp\left(\int_\Sigma \mu\right)$. 
This is independent of the choice of $R$ and $\nabla_R$. 
Similarly, if $\chi \colon \Sigma \to M$ and $(P, Y)$ is a bundle gerbe 
over $M$ with connective data $(\nabla, f)$, we can define the {\em holonomy} of 
$(\nabla, f)$ over $\chi$, $\hol(\nabla, f, \chi)$, to be the holonomy 
of the pullback of $(\nabla, f) $ by $\chi$. Notice that the holonomy 
depends on the choice of curving. It is a straightforward calculation that if bundle gerbes over $M$ are $D$-stably isomorphic they have the same holonomy.

\section{Cup product bundle gerbes}\label{section cup product bundle gerbes}
\subsection{The cup product bundle gerbe construction} Our 
aim is to show that, {after pulling back} by the Weyl map $p\colon T \times SU(n)/T \to SU(n)$, the basic bundle
gerbe decomposes into a product of simpler objects. In this section, we define these simpler objects (namely \textit{cup product bundle gerbes}), and consider more generally the class of bundle gerbes that decompose into such objects (namely \textit{general cup product bundle gerbes}).

In \cite{stuartthesis}, Johnson constructed the \textit{cup product bundle gerbe} 
$(f\cup L, f^{-1}(\R))$ over $M$ from a smooth map $f:M\to S^1$ and a 
line bundle $L\to M$. The motivating idea for this construction 
was that the Dixmier-Douady class of the 
cup product bundle gerbe should be the cup product of the winding class of $f$ and the chern class of $L$. 
The definition is as follows: let $\RR \to \RR/ \ZZ= U(1)$ and note that $f^{-1}(\RR) \to M$ is a surjective submersion, in fact a principal $\ZZ$-bundle.
As a result there is a well-defined map $d \colon f^{-1}(\RR)^{[2]} \to \ZZ$  given by $d(x, y) = y - x$ and 
satisfying $\delta(d) = 0$, 
in the sense of Section~\ref{subsec:surj submer}. 
{Johnson's} cup product bundle gerbe is then defined 
by $(f\cup L)_{(m, x, y)} = L^{d(x, y)}_{m} = L_m^{(y - x)}$.
Tensor product gives rise to a bundle gerbe product 
via the obvious identification $L^{(y-x)} \otimes L^{(z-y)} = L^{(z-x)}$. Note that if $n < 0$ then 
we define $L^n = (L^*)^{-n}$ and $L^0$ is the trivial bundle. 

For our purposes we need more general surjective submersions.  Notice first that if $K \to X$ is a 
Hermitian line bundle and $h \colon X \to \ZZ$ is a smooth function (that is, locally constant), then there is a Hermitian line 
bundle $K^h \to X$ defined fibrewise by 
\begin{equation*}
(K^h)_x = (K_x)^{h(x)}
\end{equation*}
for every $x \in X$. In other words, 
$K^h$ on each connected component of $X$ is just $K$ raised to the tensor power determined by the 
constant value of $h$ on that component.  

We consider $Y \to M$ a surjective submersion, $g \colon Y^{[2]} \to \ZZ$
with $\delta(g) = 0$ and $L \to M$ a line bundle.  Let  $L^g \to  Y^{[2]}$ be $(\pi^{[2]})^{-1}(L)^g $ where $\pi^{[2]} \colon Y^{[2]} \to M$. We will often abuse notation like this and omit obvious 
projections.
 The existence of the bundle gerbe product follows from the fact that $\delta(g) = 0$. 
We have the following definition. 
\begin{definition}\label{definition cup product} 
Let $Y\to M$ be a surjective submersion, $L\to M$ be a line bundle, and $g:Y^{[2]}\to \Z$ be a smooth map 
satisfying $\delta(g) = 0$. The bundle gerbe $(L^g, Y)$ over $M$ is the \textit{cup product bundle gerbe over $M$ of $L\to M$ and $g:Y^{[2]}\to \Z$}. 
\end{definition}

Any use of the term `cup product bundle gerbe' henceforth shall refer to Definition \ref{definition cup product}, rather than Johnson's definition mentioned above. More generally we have the following definition:

\begin{definition} 
\label{def:general-cup-product}
If $(L_i^{g_i}, Y)$ are cup product bundle gerbes over $M$ for $i = 1, \dots, n$,  we call  the reduced product $\otimes_R (L_i^{g_i}, Y)$ the {\em general cup product bundle gerbe of $L_i\to M$ and $g_i: Y^{[2]}\to \Z$ for $i=1, \dots, n$}.
\end{definition}
Associated to $g \colon Y^{[2]} \to \ZZ$ is a 
class in $H^1(M, \ZZ)$. This is  defined  by {using the fundamental complex} \eqref{eq:fund-comp} {to solve} $\delta(\psi) = g$ for 
some $\psi \colon Y \to \RR$ and taking the class to be the winding class of 
the map $q: M\to U(1)$ whose value at $m\in M$ is $\exp(2 \pi i \psi(y))$, 
where $\pi(y)=m$.  This {class} is represented in de Rham cohomology by $\frac{1}{2 \pi i} q^{-1} dq$. A straightforward calculation shows that 
the Dixmier-Douady class of the cup product bundle gerbe 
$(L^g, Y)$ is the cup product of this class and the chern class of $L$. Hence it is 
decomposable in the sense of \cite{VarMelSin}.  More generally, the 
Dixmier-Douady class of a general cup product bundle gerbe is a sum of such cup products.

\subsection{Stable isomorphisms of general cup product bundle gerbes}
In this section we will consider sufficient criteria for cup product bundle gerbes and general cup product bundle gerbes to be stably isomorphic. These results, in particular Corollary \ref{cormaincupproductresult}, will simplify calculations in Section \ref{sec:stable}.  The proofs of the following are straightforward:
\begin{proposition}
Let $\left(L^g, Y \right)$ be a cup product bundle gerbe. If there exists a smooth map $h: Y\to \Z$ such that $g = \delta(h)$, then $\left(L^g, Y  \right)$ is trivialised by $L^h \to Y$. 
\end{proposition}

\begin{corollary}
Let $\left(L^f, Y\right)$ and $(L^g, X)$ be cup product 
bundle gerbes over $M$. If there exists a smooth map 
$h: X\times_M Y\to \Z$ such that $f - g = \delta(h)$, 
then $$\left(L^f, Y\right) \cong_{\textup{stab}} (L^g, X)$$ 
with trivialising line bundle $L^h\to X\times_M Y $.
\end{corollary}

\begin{corollary}
\label{cormaincupproductresult}
Let $(L_i^{f_i}, Y)$ and $(L_i^{g_i}, X)$ be cup product bundle 
gerbes over $M$ for $i=1, \dots, n$. If there exist smooth maps $h_i: X\times_M Y \to \Z$ for each $i= 1, \dots, n$ 
satisfying $f_i - g_i=  \delta(h_i)$ for all $i = 1, \dots, n$, then
 \begin{align*}
 \sideset{}{_\mathrm{R}}\bigotimes_{i=1}^n  (L_i^{f_i}, Y) 
 \cong_{\textup{stab}} \sideset{}{_\mathrm{R}}\bigotimes_{i=1}^n (L_i^{g_i}, X)
 \end{align*}
 with trivialising line bundle $\Motimes_{i} L_i^{h_i} \to X\times_M Y$.
\end{corollary}

\subsection{Geometry of cup product bundle gerbes}
\label{teqwer} 

We next describe connective data $(\nabla, f)$ and compute 
the associated three-curvature $\omega$ on a cup product 
bundle gerbe. The induced connective data on a general 
cup product bundle gerbe can then be easily inferred. 

	Let $(L^g, Y)$ be a cup product bundle gerbe over $M$ and 
  $\nabla$ be a connection on $L\to M$ with curvature $F_\nabla$. 
	Then $L^g \to Y^{[2]}$ restricted to each connected component of $Y^{[2]}$ is a tensor power 
	of $L$ (pulled back to $Y^{[2]}$) with the power determined by the corresponding value of $g  \colon Y^{[2]} \to \ZZ$. 
	Taking appropriate tensor powers of $\nabla$ gives a 
  natural connection for  $L^g$ which we denote by $\nabla^g$. It is easy 
	to check that this is a bundle gerbe connection and that it has curvature:
		 $$
	 F_{\nabla^{g}} = g\, {\pi^{[2]}}^{*}F_{\nabla}.
	$$ 

To construct a curving for this connection we need a small amount of additional data as follows:

\begin{proposition} 
\label{curvaturetheorem1}
	Let $(L^g, Y, \pi)$ be a cup product bundle gerbe over $M$, $\nabla$ be a connection on $L\to M$, and $\nabla^{g}$ be the bundle gerbe connection defined above. Then 
	\begin{enumerate}[(1),font=\upshape]
		\item there exists a smooth function $\varphi:Y  \to \R$ such that $\delta(\varphi) = g;$
		\item the $2$-form $f\in \Omega^2(Y )$ defined by 
		$$
		f =  \varphi \, \pi^*F_\nabla
		$$ 
		satisfies $\delta(f) = F_{\nabla^{g}},$ so $f$ is a curving for the  connection $\nabla^{g};$
		\item 	if  $q \colon M \to U(1)$ is defined by $q = \exp(2 \pi i \varphi),$ the 
 three-curvature $\omega \in \Omega^3(M )$ of $(\nabla^{g}, f)$ is given by 
		$$
		\omega =  \frac{ q^{-1} d q}{2 \pi i}  \wedge F_{\nabla}; 
		$$ 
		\item the real DD-class of $(L^g, Y )$ is represented by
		$$
		-\frac{1}{4 \pi^2}  q^{-1} d q \wedge F_{\nabla}.
		$$ 
	\end{enumerate}
\end{proposition}

\begin{proof} The existence of  $\varphi$  follows by exactness of the fundamental complex and $\delta(g)=0$. For $(2)$ we have
	$$
	\delta(f) =  \delta(\psi) {\pi^{[2]}}^*(F_\nabla) = g  {\pi^{[2]}}^*(F_\nabla) = F_{\nabla^{g}}.
$$
Equations (3) and (4) follow from definitions. 
\end{proof}

Similarly in the general cup product gerbe case we have the following Proposition. 

\begin{proposition} 
\label{curvaturetheorem2}
	For each $i = 1, \dots, n$ let  $(L_i^{g_i}, Y, \pi)$ be a  cup product bundle gerbe over $M$,  $\nabla_i$ be a connection on $L_i\to M$, and $\nabla_i^{g_i}$ be the corresponding bundle gerbe connection defined above. Then 
	\begin{enumerate}[(1),font=\upshape]
		\item there exist  smooth functions $\varphi_i:Y  \to \R$ such that $\delta(\varphi_i) = g_i;$
		\item the $2$-form $f\in \Omega^2(Y )$ defined by 
		$$
		f =  \sum_{i=1}^n \varphi_i \, \pi^*F_{\nabla_i}
		$$ 
		satisfies $\delta(f) = \sum_{i=1}^n F_{\nabla_i^{g_i}},$ so $f$ is a curving for the  product connection induced by the $\nabla_i^{g_i};$
		\item 	if  $q_i \colon M \to U(1)$ is defined by $q_i = \exp(2 \pi i \varphi_i),$ the 
 three-curvature $\omega \in \Omega^3(M )$ of the general cup product gerbe of the $(L_i^{g_i}, Y, \pi)$ is given by 
		$$
		\omega = \sum_{i=1}^n   \frac{ q_i^{-1} d q_i}{2 \pi i}  \wedge F_{\nabla_i}; 
		$$ 
		\item the real DD-class of the general cup product bundle gerbe of the $(L_i^{g_i}, Y, \pi)$ is represented by
		$$
		-\frac{1}{4 \pi^2} \sum_{i=1}^n  q_i^{-1} d q_i \wedge F_{\nabla_i}.
		$$ 
	\end{enumerate}
\end{proposition}

%
%
\section{Cup product bundle gerbes over $T \times SU(n)/T$}
\label{sec:weyl-bundle-gerbe}

\subsection{The $i$-th cup product bundle gerbes} In this section, 
we will define cup product bundle gerbes over $T\times \su$ 
called the \textit{$i$-th cup product bundle gerbes}, for $T$ 
the subgroup of $SU(n)$ consisting of diagonal matrices.  
Our aim is to construct their reduced product, which we call 
the {\em Weyl bundle gerbe}.  We begin with some preliminaries. 
For $n\in \N$, let $\text{Proj}_n$ be the set of $n$-tuples of 
orthogonal projections $(P_1, \dots, P_n)$, where, for each $i$, 
$P_i:\C^n\to W_i$ for $W_i$ mutually orthogonal one-dimensional subspaces of $\C^n$.
It follows from the characterisation of homogeneous spaces 
\cite[Theorem 21.18]{LEE} that there is a bijection 
$\su\cong\textup{Proj}_n$, which implies $\text{Proj}_n$ is a smooth manifold diffeomorphic to $\su$. 

	For each $i = 1, \dots, n$, let $p_i \colon T \to S^1$ be the homomorphism  $p_i(t_1, \dots, t_n) = t_i$. 
	Define $J_i \to \su$ to be the (homogeneous) Hermitian line bundle 
	associated to the principal $T$-bundle $SU(n) \to \su$ via
	the action of $T$ on $\CC$ by $t \cdot z = p_i(t^{-1}) z$. Define $K_i\to \text{Proj}_n$, a subbundle of the trivial bundle $\CC^n \times \text{Proj}_n$,  by $(K_i)_{(P_1, \dots, P_n)} = \text{im}(P_i)\times \{(P_1, \dots, P_n)\}$. It can be verified easily that $J_i$ and $K_i$ are $SU(n)$-equivariant line bundles with respect to the $SU(n)$-action on $\su$ defined by left multiplication and the $SU(n)$-action on $\text{Proj}_n$ defined by $$g\cdot (v, P_1, \dots, P_n) = (gv, gP_1g^{-1}, \dots, gP_ng^{-1}).$$ By the equivalence of linear representations and equivariant line bundles, there is an $SU(n)$-equivariant isomorphism $J_i\cong K_i$. Throughout this work, we will continue to write $J_i\to \su$, but will in practice work with the line bundles $K_i\to \text{Proj}_n$. 

Denote by  $X_T$ the subset of $(x_1, \dots, x_n) \in \RR^n$ which sum to zero.  Define a surjective
submersion $\pi \colon X_T \to T$ by 
$$
\pi( x_1, \dots, x_n) = \text{diag}(e^{2\pi i x_1}, \dots, e^{2\pi i x_n}).
$$
Note that this defines a surjective submersion $\pi_c \colon X_T \times SU(n)/T \to T \times \su$
and that $\left( X_T \times \su \right)^{[2]} = X_T^{[2]} \times \su$. 
Define $d_i: X_T^{[2]}\to \ZZ$ by $d_i(x, y) =  x_i - y_i$ and extend it to $X_T^{[2]} \times \su$ by projection
with the same name.

 The $i$-th cup product bundle gerbe is defined as follows.
\begin{definition}
The \textit{$i$-th cup product bundle gerbe over $T\times \su$} for $i=1, \dots, n$ is the cup product bundle gerbe 
$$
\left(J_i^{d_i}, X_T\sstimes \su, {\pi}_c \right).
$$
\end{definition}

\begin{proposition}\label{cupprodisequivariant} 
The $i$-th cup product bundle gerbe is $SU(n)$-equivariant for the action of $SU(n)$ on $T\sstimes \su$ defined by multiplication in the $\su$ factor. 
\end{proposition}
\begin{proof} This follows easily by $SU(n)$-homogeneity of $J_i\to \su$ and by noting that the $SU(n)$-action on each of the spaces in the bundle gerbe is given by multiplication on the $\su$ factor.
\end{proof}

\subsection{Geometry of the $i$-th cup product bundle gerbes} 
\label{pollen} We will now apply the results from 
Section~\ref{teqwer} to the $i$-th cup product bundle gerbes. 
The following standard fact will be used repeatedly: let $L\to M$ 
be a line bundle that is a subbundle of the trivial bundle of rank 
$n$. Let $P:\C^n\times M\to L$ be orthogonal projection. Then the 
induced connection $\nabla = P\circ d$ on $L$ (for $d$ the 
trivial connection) has curvature $F_{\nabla} = \text{tr}(PdPdP).$

\begin{proposition}\label{curvaturetwoform}
	There is a canonical line bundle connection $\nabla_{{J}_i}$ on ${J}_i\to \su$ with curvature $F_{\nabla_{{J}_i}} = \textup{tr}(P_idP_idP_i)$ for $P_i:\su\times \C^n\to J_i $ orthogonal projection.
\end{proposition}

\begin{proof} By the standard fact above, we need only show that ${J}_i$ is a subbundle of the trivial bundle of rank $n$. This follows by noting that ${J}_i$ is a subbundle of the $SU(n)$-homogeneous vector bundle $({\C^n\times SU(n)})/{T}\to \su$, which is isomorphic to the trivial bundle $\C^n \times \su \to \su$ by the equivalence of linear representations and equivariant bundles. \end{proof}

\begin{proposition}\label{connectiononJ} 
	Let $\nabla_{J_i}$ be the connection on $J_i$ from \textup{Proposition \ref{curvaturetwoform}}. Let 
	$$
	\pi_c^{[2]}: X_T^{[2]}\times \su \to T\times \su
	$$
	 be projection. Then there is a bundle gerbe connection $\nabla_{c_i}$ on the $i$-th cup product bundle gerbe with curvature $$ F_{\nabla_{c_i}} = d_i\, ({\pi_c^{[2]}})^{*}\textup{tr}(P_idP_idP_i)$$ for $P_i: T\times \su\times \C^n\to J_i$ orthogonal projection. 
\end{proposition}

\begin{proof}
This follows from our discussion in Section~\ref{teqwer} and Proposition~\ref{curvaturetwoform}.
\end{proof}

\begin{proposition}
\label{connective data on ith cup product bundle gerbe} 
Let $\nabla_{c_i}$ be the connection on the $i$-th cup product bundle gerbe from \textup{Proposition \ref{connectiononJ}}. Let $P_i:T\times \su\times \C^n\to J_i$ be orthogonal projection$,$ and define a $2$-form $f_{c_i} \in \Omega^2\left(X_T\times \su\right)$ by $$f_{c_i}(t, gT, x) = -x_i\, \pi_c^* \textup{tr}(P_idP_idP_i).$$ Abuse notation and denote the pullback of $p_i$ to $T\times \su$ by $p_i$. Then 	\begin{enumerate}[(1),font=\upshape]
		\item the $2$-form $f_{c_i}$ satisfies $\delta(f_{c_i}) = F_{\nabla_{c_i}},$ so $f_{c_i}$ is a curving for $\nabla_{c_i};$
		\item the three-curvature $\omega_{c_i} \in \Omega^3\left(T\times\su \right)$ of $(\nabla_{c_i}, f_{c_i})$ is given by $$\omega_{c_i} = -\frac{1}{2\pi i}\, p_i^{-1}dp_i  \textup{tr}(P_idP_idP_i);$$
		\item the real DD-class of the $i$-th cup product bundle gerbe is represented by
		$$\frac{1}{4\pi^2}\, p_i^{-1}dp_i \textup{tr}(P_idP_idP_i). $$
	\end{enumerate}
\end{proposition}
\begin{proof}
Consider the proof of  Proposition \ref{curvaturetheorem1}.  If $\varphi(y) = -y_i$ then $\delta(\varphi)(x, y) = \varphi(y) - \varphi(x)  = x_i - y_i = d_i(x, y)$ and  $q(x) = \exp(- 2 \pi i x_i) = p_i^{-1}$. The results
then follow by substitution into the formula in  Proposition \ref{curvaturetheorem1} 
\end{proof}

By comparing this result to Proposition \ref{curvaturetheorem1},
 we see that the real Dixmier-Douady class of the $i$-th cup product bundle gerbe is  the cup product of the winding class of the map $p_i^{-1} \colon T \to U(1)$ and the chern class of the
line bundle $J_i \to \su$.

\subsection{The Weyl bundle gerbe}
 We can now define the \textit{Weyl bundle gerbe}, and compute its connective data and associated three-curvature using results from Section~\ref{pollen}.
 \begin{definition}\label{definitionofthefinalcupproductbundlegerbe} 
The \textit{Weyl bundle gerbe} over $T\times \su$ is the reduced product of the $i$-th cup product bundle gerbes, denoted  $$\left(P_c, X_T \times \su, \pi_c \right) := \sideset{}{_\mathrm{R}}\bigotimes_{i=1}^n \left( J_i^{d_i}, X_T\sstimes \su\right).$$ 
\end{definition}
\begin{proposition}
The Weyl bundle gerbe is $SU(n)$-equivariant for the action of $SU(n)$ on $T\times \su$ defined by multiplication in the $\su$ factor. 
\end{proposition}
\begin{proof}
This follows from Proposition \ref{cupprodisequivariant} and 
the fact that the reduced product of equivariant bundle gerbes is again equivariant. 
\end{proof}

 We compute connective data and the curvature of the Weyl bundle gerbe as follows.

\begin{proposition}\label{jello} 
	Let $P_i: \su\sstimes T \sstimes \C^n \to {J}_i$ be orthogonal projection and $\pi^{[2]}: X^{[2]}_T \sstimes \su \to T\sstimes \su$ be projection. The $i$-th cup product bundle gerbe connections $\nabla_{c_i}$ from \textup{Proposition \ref{connectiononJ}} induce a bundle gerbe connection $\nabla_c$ on  the Weyl bundle gerbe with curvature  $$F_{\nabla_c}=  \sum_{i=1}^n d_i (\pi^{[2]})^{*}\textup{tr}(P_idP_idP_i).$$
\end{proposition} 
\begin{proof}
	This follows from elementary bundle gerbe theory and Proposition \ref{connectiononJ}.
\end{proof}

Using Proposition \ref{connective data on ith cup product bundle gerbe} and elementary facts about 
products of bundle gerbe connections and curvings we similarly obtain the connective data
on the Weyl bundle gerbe.

\begin{proposition}\label{corollary data on cup prod} 
Let $\nabla_{c}$ be the connection on the Weyl bundle gerbe from \textup{Proposition \ref{jello}}. Let $P_i: T\sstimes \su\sstimes \C^n\to J_i$ be orthogonal projection, and define a $2$-form $f_c \in \Omega^2\left(X_T \times SU(n)/T\right)$ by \begin{align}\label{teaw} f_c(x_1, \dots, x_n, gT) := -\sum_{i=1}^n x_i  \pi_c^*\textup{tr}(P_idP_idP_i).\end{align} Abuse notation and denote the pullback of $p_i$ to $T\sstimes \su$ by $p_i$. Then  \begin{enumerate}[(1),font=\upshape]
		\item the $2$-form $f_{c}$ satisfies $\delta(f_c) = F_{\nabla_c},$ so $f_c$ is a curving for $\nabla_c;$
		\item the three-curvature $\omega_c \in \Omega^3\left(T\sstimes \su\right)$ of $(\nabla_{c}, f_c)$ is given by $$\omega_c = -\frac{1}{2\pi i}\sum_{i=1}^n p_i^{-1} dp_i \textup{tr}(P_idP_idP_i);$$
		\item the real DD-class of $\left(P_c, X_T\right)$ is represented by $$\frac{1}{4\pi^2}\sum_{i=1}^n p_i^{-1} d p_i  \textup{tr}(P_idP_idP_i).$$ 
	\end{enumerate}
\end{proposition}

%
%

\section{The basic bundle gerbe and the Weyl map} 
\label{sec:basic-pullback}

\subsection{The basic bundle gerbe}\label{section: definition of basic bg} We review the  construction of the basic bundle gerbe over $SU(n)$ in \cite{53UNITARY}.  Our aim is to show that, when pulled back to $T \times \su$ by the Weyl map,
		the basic bundle gerbe is a general cup product bundle gerbe different to the one defined in \ref{definitionofthefinalcupproductbundlegerbe}.  We will then exploit the techniques from the first section to 
	construct a stable isomorphism between them.
	
Let $Z := U(1)\backslash \{1\}$
 and define the manifold $$Y:= \left\{(z, g) \in Z\times SU(n) \ | \ z\notin \text{spec}(g)\right\}.$$ Let $\pi_b: Y\to SU(n)$ be the surjective submersion defined by projection onto the second factor. Note that $Y^{[2]}$ can be identified with triples $(z_1, z_2, g)$ with $z_1, z_2\notin \text{spec}(g)${, the set of eigenvalues of $g$}.  Order the set $Z$ by anti-clockwise rotation and define the following
 notions. 
	\begin{definition}
		Let $(z_1, z_2, g)\in Y^{[2]}$ and $\lambda$ be an 
    eigenvalue of $g$. Say that $\lambda \in Z$ is 
    \textit{between} $z_1$ and $z_2$ if $z_1 < \lambda < z_2$ 
    or $z_2 < \lambda < z_1$. Call $(z_1, z_2, g)\in Y^{[2]}$ 
    \textit{positive} if there exist eigenvalues of $g$ 
    between $z_1 > z_2$, \textit{null} if there are no 
    eigenvalues of $g$ between $z_1$ and $z_2$, and 
    \textit{negative} if there exist eigenvalues of $g$ between $z_1<z_2$. 
	\end{definition}
Denote the set of all positive, null, and negative triplets in 
$Y^{[2]}$ by $Y^{[2]}_+, Y^{[2]}_0$ and $Y^{[2]}_{-}$ respectively. 
Note that $(z_1, z_2, g)\in Y_+^{[2]}$ if and only if 
$(z_2, z_1, g)\in Y_-^{[2]}$. Elements in each of these sets 
are depicted in Figure \ref{fig:circleimage2}, where we assume 
for simplicity that all eigenvalues of $g$ are in the connected 
component of $Z\backslash \{z_1, z_2\}$ containing $\lambda$. 
	
	\begin{figure}[h] \centering 
		\begin{tikzpicture}[scale=0.7]  
		\draw (0, 0) circle (1.5 cm);
		\draw[fill = black] (1, 1.118) circle (2pt);
		\draw[fill = black] (0, -1.5) circle (2pt);
		\draw[fill = white, draw = black] (1.5, 0) circle (2pt);
		\node[] at (1.28, 1.3) {$\ z_1$};
		\node[] at (0, -1.85) {$z_2$};
		\draw[fill = black] (-1, 1.118) circle (2pt);
		\node[] at (-1.2, 1.35) {$\lambda \ \  $};
		\node[] at (0, -2.7) {$(z_1, z_2, g) \in Y_-^{[2]}$};

		\draw (6, 0) circle (1.5 cm);
		\draw[fill = black] (7, 1.118) circle (2pt);
		\draw[fill = black] (4.585, -.5) circle (2pt);
		\draw[fill = black] (5, 1.118) circle (2pt);
		\draw[fill = white, draw = black] (7.5, 0) circle (2pt);
		\node[] at (7.28, 1.3) {$\ z_2$};
		\node[] at (4.3, -.6) {$\lambda$};
		\node[] at (4.8, 1.35) {$z_1\ \ $};
		\node[] at (1, 1) {};
		\node[] at (6, -2.7) {$(z_1, z_2, g) \in Y_0^{[2]}$};
		
		\draw (12, 0) circle (1.5 cm);
		\draw[fill = black] (12, -1.5) circle (2pt);
		\draw[fill = black] (12, 1.5) circle (2pt);
		\draw[fill = black] (12.9, 1.2) circle (2pt);
		\draw[fill = white, draw = black] (13.5, 0) circle (2pt);
		\node[] at (12, 1.85) {$\lambda $};
		\node[] at (12, -1.85) {$\ z_1$};
		\node[] at (13.17, 1.37)  {$\ z_2$};
		\node[] at (1, 1) {};
		\node[] at (12, -2.7) {$(z_1, z_2, g) \in Y_+^{[2]}$};
		\end{tikzpicture}	\captionof{figure}{Components of $Y^{[2]}$ }
		\label{fig:circleimage2} \end{figure}
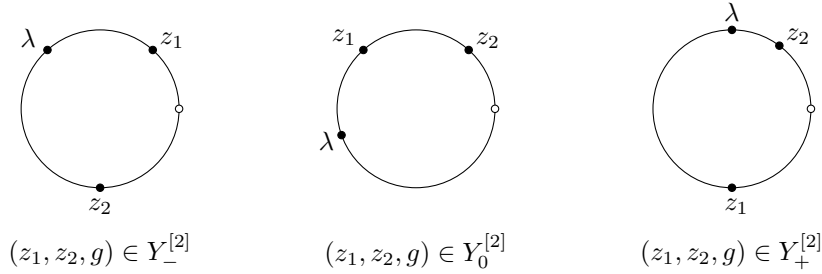

	 We define a {Hermitian line bundle}
   $P_b\to Y^{[2]}$ as follows. For $\lambda$ an 
   eigenvalue of $g\in SU(n)$, let $E_{(g, \lambda)}$ 
   denote the $\lambda$-eigenspace of $g$. Define the 
   vector bundle $L\to Y^{[2]}_{+}$ fibrewise by 
   \begin{equation}
   \label{eq:basic-bundle}
   L_{(z_1, z_2, g)} = \Moplus_{z_1>\lambda > z_2} E_{(g, \lambda)}.
   \end{equation}
	For a proof that this is indeed a vector bundle 
  see \cite{53UNITARY}. Note that $L_{(z_1, z_2, g)}$ 
  has finite dimension as a finite sum of finite-dimensional 
  spaces. Therefore we can define 
  \[
  (\Pbasic)_{(z_1, z_2, g)} = \begin{cases} \text{det}(L_{(z_1, z_2, g)}) &\text{if } (z_1, z_2, g)\in Y^{[2]}_+\\ 
	\C&\text{if } (z_1, z_2, g)\in Y^{[2]}_0 \\
	\text{det}(L_{(z_2, z_1, g)})^* & \text{if } (z_1, z_2, g) \in Y^{[2]}_{-}. \end{cases} 
  \]
	By \cite{53UNITARY}, $\Pbasic \to Y^{[2]}$ is a 
  smooth locally trivial Hermitian
 line bundle, and 
  there is an associative multiplication operation  
  endowing $(P_b, Y, SU(n))$ with a bundle gerbe {structure}.
	\begin{definition}\label{bbbbbbg}
		Call the bundle gerbe $\basicbg$ over $SU(n)$ constructed 
    above the \textit{basic bundle gerbe over $SU(n)$}, or simply the \textit{basic bundle gerbe}.
	\end{definition}

\subsection{The pullback of the basic bundle gerbe by the Weyl map} 

	Recall that, for $G$ a compact, connected Lie group and $T$  a maximal torus of $G$, the \textit{Weyl map}\footnote{So-called
	because it is used in the Weyl integral formulae in \cite{weyl1}. } is the $G$-equivariant map defined by
	 \begin{align*}
	p: T\times {G}/{T} \to G, \ (t, gT)\mapsto gtg^{-1}.
	\end{align*} 
	
	The Weyl map has a number of attractive features for our purposes. Firstly, the action of $SU(n)$ by conjugation on itself lifts to an action of $SU(n)$ on $T \times \su$ 
	where it acts only on the left of $\su$.  Secondly, if $g \in SU(n)$, we can decompose $\CC^n$ into a direct-sum of distinct eigenspaces of $g$.
	These eigenspaces may change in dimension as $g$ varies and thus do not extend to vector bundles over the whole of $SU(n)$.  However, on $T \times \su$ things are 
	much more pleasant.  If we consider $(t, hT)$ which maps to $g = h t h^{-1}$, then we can write $\CC^n$ as a direct sum of the one-dimensional 
	spaces which are multiples of the standard basis vector or eigenspaces of $t$.  Moreover, if we act on these by $h$, we  decompose $\CC^n$ into a direct sum of one-dimensional spaces which are subspaces of the eigenspaces of $g$.  In fact, these one-dimensional spaces are a decomposition of the trivial $\CC^n$ 
	bundle over $T \times \su$ into homogeneous vector bundles $J_1 \oplus J_2 \oplus \dots \oplus J_n$ pulled back from $\su$.   We will make extensive
	use of these basic geometric facts.

		It is straightforward to see that under the identification $\su \cong \text{Proj}_n$ the Weyl map 
	$p:T\times \text{Proj}_n\to SU(n)$ is given by \begin{align*} p: (t, P_1, \dots, P_n)\mapsto \sum_{i=1}^np_i(t)P_i. \end{align*}

	Notice that $p^{-1}(Y)$ is the collection of all $(t, z, gT) \in T \times Z \times \su$ with $z \neq t_i$ for any $i = 1, \dots, n$.
	If we let $Y_T \subset T \times Z$ be all $(t, z)$ with  $z \neq t_i$ for any $i = 1, \dots, n$ then we have 
	$$
	p^{-1}(Y) = Y_T \times \su.
	$$
	
	Our aim in this section is to prove the following Proposition for particular $\varepsilon_i$ defined
	below in Definition \ref{defnepsilon}, {thereby realising the pullback of the basic bundle gerbe as a general cup product bundle gerbe}.
		
		\begin{proposition}\label{pullbackstableisomainresult} 
		There is an $SU(n)$-equivariant isomorphism over $T\times \su$ $$\pullbackbasicbg \cong_{SU(n)} \sideset{}{_\mathrm{R}}\bigotimes_{i=1}^n \left( J_i^{\varepsilon_i}, Y_T\sstimes \su\right).$$
	\end{proposition}

\noindent \noindent The proof of Proposition \ref{pullbackstableisomainresult} relies on the following intermediary isomorphisms over $T\times \su$:  \begin{align*} \pullbackbasicbg \ &\stackrel{\text{Prop \ref{iso1}  }}{\cong_{SU(n)}   }    \left(  \PbasicT \times_T SU(n), Y_T\sstimes \su \right)\\
	&\stackrel{\text{Prop \ref{exactly} }}{\cong_{SU(n)}} \left(\Motimes_{i=1}^nJ_i^{\varepsilon_i}, Y_T\sstimes \su \right).
	\end{align*}

We begin with the following proposition and leave the proof of this result as an exercise. Let  $\PbasicT := (P_b)|_{Y_T^{[2]}}$. The restriction of the basic bundle gerbe to $T$ is $\basicbgtorus$. 
	\begin{proposition}[{\cite[p.$\,$1582]{53UNITARY}}] 
		Define $\PbasicT\times_T SU(n)$ to be the set of equivalence classes in $\PbasicT\times SU(n)$ under the relation $$(v_1\wedge \cdots \wedge v_k, g) \sim (tv_1 \wedge \cdots \wedge tv_k, gt^{-1})$$ for all $t\in T$
		where $k$ is the rank of $L$ in \eqref{eq:basic-bundle}. Then ${\PbasicT \times_T SU(n)}$ is a line bundle over $Y_T^{[2]}\times \su$, and there is an associative multiplication induced by that on $(P_{b, T}, Y_T)$ making \begin{align}\label{vege} \left(  \PbasicT \times_T SU(n), Y_T\sstimes \su \right)\end{align} an $SU(n)$-equivariant bundle gerbe over $T\times \su$ with respect to the $SU(n)$-action on $T\times \su$ defined by multiplication on the $\su$ component.
	\end{proposition}

	The following proposition is shown in \cite[Proposition 7.1]{53UNITARY}.
	\begin{proposition}[{\cite[Proposition 7.3]{53UNITARY}}] \label{iso1} 
		There is an $SU(n)$-equivariant bundle gerbe isomorphism over $T\times \su$ $$ \left(  \PbasicT \times_T SU(n), Y_T\sstimes \su \right)\cong_{SU(n)} \pullbackbasicbg.$$  
	\end{proposition} 

	Recall that $T$ is the maximal torus of $SU(n)$ consisting of diagonal matrices, and $p_i:T\to S^1$ is the homomorphism sending $t\in T$ to its $i$-th diagonal. To define $\varepsilon_i$ we use the ordering on $Z$ from Section \ref{section: definition of basic bg}. Let $i\in \{1, \dots, n\}$ throughout.  
		\begin{definition}
		\label{defnepsilon}
		Define $\varepsilon_i: Y_T^{[2]}\to \Z$ by $$\varepsilon_i(z_1, z_2, t) = \begin{cases}
		1 & \text{if } z_1 > p_i(t) > z_2 \\
		-1 & \text{if } z_2 > p_i(t) > z_1 \\ 
		0 & \text{otherwise}.
		\end{cases} $$
	\end{definition}
	
	Notice that $\varepsilon_i$ is a smooth function (that is, locally constant) on $Y_T^{[2]}$. 
	
		\begin{definition}
	Let $\C_{p_i}$ be the space $\C$ equipped with the $T$-action $v\cdot t := p_i(t^{-1})v$. 
	\end{definition}
Throughout this section, let $\C_{p_i}^1:=\C_{p_i}, \C_{p_i}^{-1}:= \C_{p_i}^*,$ and $\C_{p_i}^0:=\C$, where $\C_{p_i}^0$ is equipped with the identity action. The space $\C_{p_i}^*$ can be understood as the dual of $\C_{p_i}$, or equivalently as the space $\C$ equipped with the dual action $v\cdot t = p_i(t)v$. 

	Recall that $J_i\to \su$ is the $SU(n)$-homogeneous line bundle defined by setting $J_i  := \C\times_{p_i} SU(n)$ under the relation $(z, s) \sim_{p_i} (p_i(t^{-1})z, st)$ for all $t\in T$. 
\begin{proposition}
	There is an associative multiplication making \begin{align} \label{biscut}\left(J_i^{\varepsilon_i}, Y_T\times \su \right)\end{align} a cup product bundle gerbe over $T\times \su$. Moreover, this bundle gerbe is $SU(n)$-equivariant with respect to the $SU(n)$-action on $T\times \su$ defined by multiplication on the $\su$ component.
\end{proposition}
\begin{proof}
	To see that this is a cup product bundle gerbe, and 
  hence a bundle gerbe, it suffices to show that 
  $\varepsilon_i: Y_T^{[2]}\to \Z$ satisfies the 
  cocycle condition $\delta(\varepsilon_i) = 0$. This 
  is trivial to verify on the connected components of 
  $Y_T^{[2]}$. The equivariance result follows easily.
\end{proof}

\begin{proposition}\label{exactly}
There exists an $SU(n)$-equivariant bundle gerbe isomorphism over $T\times \su$ \begin{align*} \left(  \PbasicT \times_T SU(n), Y_T\sstimes \su\right)\cong_{SU(n)} \left( \Motimes_{i=1}^nJ_i^{\varepsilon_i}, Y_T\sstimes \su \right).\end{align*}
\end{proposition}
\begin{proof}
First, we show there is a line bundle isomorphism $\PbasicT \cong \Motimes \C_{p_i}^{\varepsilon_i} \times Y_T^{[2]}$. 
Let $(z_1, z_2, t)\in Y_T^{[2]}$ with $z_1>z_2$. Suppose there are eigenvalues of $t$ between $z_1$ and $z_2$. Denote these eigenvalues by $p_{k_1}(t), \dots, p_{k_m}(t)$ for $1\leq k_1\leq \cdots \leq k_m \leq n$. Then \begin{align} \label{choc} L_{(z_1, z_2, t)} = E_{(t, p_{k_1}(t))} \oplus \cdots \oplus E_{(t, p_{k_m}(t))} \end{align}	and $$
(\PbasicT)_{(z_1, z_2, t)} = \text{det}(L_{(z_1, z_2, t)}) = \Motimes_{z_1 >\lambda > z_2} \text{det} \left(E_{(t, \lambda)}\right)=E_{(t, p_{k_1}(t))}\otimes \cdots \otimes E_{\left(t,p_{k_m}(t)\right)}.
$$ 	Each eigenspace $E_{(t, p_k(t))} \cong \C$ is equipped with a $T$-action $v\cdot s := p_k(s^{-1}) v$, hence $E_{(t, p_k(t))}\cong \C_{p_k}$ for each $k$. Since $\varepsilon_{k_i}(z_1, z_2, t)=1$ for $i=1, \dots, m$ and $\varepsilon_k=0$ otherwise $$ \left(\PbasicT\right)_{(z_1, z_2, t)} \cong \C_{p_{k_1}} \otimes \cdots \otimes \C_{p_{k_m}} \cong \C_{p_1}^{\varepsilon_1(z_1, z_2, t)}\otimes \cdots \otimes \C_{p_n}^{\varepsilon_n(z_1, z_2, t)}.$$ 
By almost identical arguments, this holds over the other components of $Y_T^{[2]}$. Therefore we have an isomorphism $\PbasicT \cong \Motimes_{i=1}^n \C_{p_i}^{\varepsilon_i} \times Y_T^{[2]},$ as claimed. This implies we have an isomorphism \begin{align}\label{nope}\PbasicT\times_T SU(n) \cong \left(\Motimes \C_{p_i}^{\varepsilon_i} \times Y_T^{[2]}\right)\times_TSU(n),\end{align} where the latter line bundle is $SU(n)$-equivariant with $T$-action defined by \begin{align}\label{tere}(z_1, \dots, z_n, u, g) \cdot t= (p_1(t^{-1})z_1, \dots, p_n(t^{-1})z_n, u, gt).\end{align} It can be verified that the line bundle isomorphism (\ref{nope}) is $SU(n)$-equivariant. This will act as our `intermediary isomorphism'. Next, consider $\left(\Motimes_{i=1}^n \C_{p_i}^{\varepsilon_i} \right) \times_T SU(n)$, the space of equivalence classes under the $T$-action defined similarly to (\ref{tere}). This is an $SU(n)$-homogeneous line bundle over $\su$, and it can be verified that the natural map \begin{align}\label{pert} {\left(\C_{p_1}^{\varepsilon_1} \otimes \cdots \otimes \C_{p_n}^{\varepsilon_n}\times Y_T^{[2]}\right) \times_T SU(n) } \to Y_T^{[2]}\times \left(\C_{p_1}^{\varepsilon_1} \otimes \cdots \otimes \C_{p_n}^{\varepsilon_n} \times_T SU(n)\right) \end{align} is a well-defined, $SU(n)$-equivariant line bundle isomorphism over $Y_T^{[2]} \times \su$. It follows by the equivalence of linear representations and equivariant bundles that there are $SU(n)$-equivariant line bundle isomorphisms \begin{align*} \C_{p_1}^{\varepsilon_1} \otimes \cdots \otimes \C_{p_n}^{\varepsilon_n} \times_T SU(n) &\cong (\C_{p_1}^{\varepsilon_1}\times_TSU(n)) \otimes \cdots \otimes (\C_{p_n}^{\varepsilon_n} \times_T SU(n))\\  &\cong J_1^{\varepsilon_1}\otimes \cdots \otimes J_n^{\varepsilon_n}. \end{align*} This, combined with (\ref{pert}), implies there is an $SU(n)$-equivariant line bundle isomorphism $$\left(\Motimes_{i=1}^n \C_{p_i}^{\varepsilon_i} \times Y_T^{[2]} \right) \times_T SU(n)\cong   \Motimes_{i=1}^n J_i^{\varepsilon_i}$$ and hence, by (\ref{nope}),  we obtain an $SU(n)$-equivariant isomorphism of line bundles $P_{b, T}\times_T SU(n) \cong \Motimes_{i=1}^n J_i^{\varepsilon_i}.$ It remains to show that this isomorphism preserves the bundle gerbe product. Suppose $(z_1, z_2, z_3, g) \in Y^{[3]}$ with $z_1 > z_2 > z_3$, and that there are eigenvalues of $g$ between $z_1$ and $z_2$ and also between $z_2$ and $z_3$. Then $L_{(z_1, z_2, g)}\oplus L_{(z_2, z_3, g)}  =L_{(z_1, z_3, g)}$ and the basic bundle gerbe product is induced from $$\text{det}(L_{(z_1, z_2, t)}) \otimes \text{det}(L_{(z_2, z_3, t)}) \cong \text{det}(L_{(z_1, z_2, t)} \oplus L_{(z_2, z_3, t)}) \cong \text{det}(L_{(z_1, z_3, t)}). $$ From the discussion above and equation (\ref{choc}), each $L_{(z_i, z_j, t)}$ decomposes into appropriate sums of powers of the $J_l$, so this becomes $$\Motimes_{i=1}^n J_i^{\varepsilon_i(z_1, z_2, t)}\otimes \Motimes_{i=1}^n J_i^{\varepsilon_i(z_2, z_3, t)} \cong \Motimes_{i=1}^n J_i^{\varepsilon_i(z_1, z_3, t)},$$ which is the cup product multiplication. The other cases proceed similarly.
\end{proof}

	Clearly, the reduced product of the $SU(n)$-equivariant bundle gerbes $(\ref{biscut})$ will be an $SU(n)$-equivariant bundle gerbe. This leads us to our final isomorphism, 
  which follows from Propositions~\ref{iso1} and~\ref{exactly}.\\

\noindent \noindent \textbf{Proposition \ref{pullbackstableisomainresult}.}\textit{
	There is an $SU(n)$-equivariant isomorphism over $T\times \su$ $$\pullbackbasicbg \cong_{SU(n)} \sideset{}{_\mathrm{R}}\bigotimes_{i=1}^n \left( J_i^{\varepsilon_i}, Y_T\sstimes \su\right).$$}

\subsection{Geometry of the pullback of the basic bundle gerbe} 
In \cite{53UNITARY},  connective data  $(\nabla_b, f_b)$ was defined on the basic bundle gerbe as follows. First, a connection on the  bundle $L$ defined in equation \eqref{eq:basic-bundle} was constructed using orthogonal 
projection of the flat connection. Taking the highest exterior power of this connection gave rise to a bundle gerbe connection $\nabla_b$ on the basic bundle gerbe.  Second, the curving 
$f_b$  was constructed using holomorphic functional calculus.  
We will not detail this construction
here as we only need the connective data on the pullback of the basic bundle gerbe which is given by  Proposition  \ref{gbgbgb} below from \cite{53UNITARY}.

Recall from Section \ref{pollen} that the  cup product bundle gerbes $(J_i^{\varepsilon_i}, Y_T\times SU(n)/T)$  can be endowed with a so-called cup product bundle gerbe connection using orthogonal projection, similar to Proposition \ref{connectiononJ}. This in turn induces a general cup product bundle gerbe connection on $\otimes_R(J_i^{\varepsilon_i}, Y_T\times SU(n)/T)$ in the obvious way. We see that this general cup product bundle gerbe connection and $\nabla_b$ are both tensor products of or the determinant of connections defined by orthogonal projection of the flat connection onto subbundles. By the naturality of these constructions it follows that  the pulled back connection on $\pullbackbasicbg$ under the isomorphism in Proposition~\ref{pullbackstableisomainresult} is the 
general cup product connection on $\otimes_R \left( J_i^{\varepsilon_i}, Y_T\sstimes \su\right)$.

	\begin{proposition}[{\cite[Appendix B]{53UNITARY}}]  \label{gbgbgb} Let $\nabla_{p^*b}$ be the connection on the pullback of the basic bundle gerbe by the Weyl map {induced by $\nabla_b$}. The pulled back curving and curvature are given by 
\begin{equation}
\label{eq:long-basic-curving}
		 f_{p^*b} = \frac{i}{4\pi} \sum_{\substack{i, k=1 \\ i\neq k}}^n (\log_z p_i - \log_z p_k + (p_k - p_i)p_k^{-1})\textup{tr}(P_i dP_kdP_k)
		 \end{equation}
		 and
\begin{multline}
\label{eq:long-basic-curvature}
\omega_{p^*b}	= 
\frac{i}{4\pi} \sum_{\substack{i, k=1 \\ i\neq k}}^n   
\left( p_i^{-1} dp_i - p_k^{-1}dp_k -p_k^{-1}dp_i +  p_k^{-1}dp_k p_k^{-1}p_i \right) \textup{tr}(P_idP_kdP_k)  \\
 \qquad    - \frac{i}{4\pi}\sum_{\substack{i, k=1 \\ i\neq k}}^n 
 p_ip_k^{-1}  \textup{tr}(dP_idP_kdP_k).
\end{multline}
\end{proposition} 
\noindent Here $\log_z$ is the branch of the logarithm defined by cutting along the ray through $z \neq 1$ and requiring $\log_z(1) = 0$.
Note that here and in  the remainder of this work, we abuse notation and let the homomorphisms $p_i$ and projections $P_i$ be defined on the spaces $Y_T\times SU(n)/T, X_T\times SU(n)/T$, or $(X_T\times_TY_T)\times SU(n)/T$ depending on the context.

The formulae in Proposition \ref{gbgbgb} can be simplified in a way that makes them more comparable to the Weyl bundle gerbe data as follows.

\begin{proposition}
\label{curvature}
Let $\nabla_{p^*b}$ be the connection on the pullback of the basic 
bundle gerbe and $\pi_{p^*b}: Y_T \times \su \to T\times \su$ be projection.  Define $\beta \in \Omega^2\left( T \times \su\right)$ by 
\begin{equation}
\notag
\beta =  -\frac{i}{4\pi}\sum_{\substack{i, k=1 \\ i\neq k}}^n  p_ip_k^{-1}\textup{tr}(P_idP_kdP_k).
\end{equation}
Then 
\begin{equation}
\label{eq:basic-curving}
f_{p^*b} =  \sum_{k=1}^n \left(\frac{-1}{2 \pi i}\log_z p_k \right) \tr(P_k dP_k dP_k) + (\pi_{p^*b})^*\beta 
\end{equation}
and consequently
\begin{equation}
\label{eq:basic-curvature}
\omega_{p^*b} =  -\frac{1}{2\pi i} \sum_{k=1}^n  p_k^{-1}dp_k  \textup{tr}(P_kdP_kdP_k) + d \beta 
\end{equation}
\end{proposition}
\begin{proof}
The proof needs a number of ingredients, some of which are proved in the Appendix.  Firstly, we know that $\sum_{k=1}^n P_k = I$ and
thus $\sum_{k=1}^n dP_k = 0$.  Also as shown in the Appendix if $i \neq k$ then $\tr(P_i dP_k dP_k) = - \tr(P_k dP_i dP_i)$. Again 
from the Appendix $\sum_{k=1}^n \tr(P_k dP_k dP_k) = 0$. Using these it is straightforward to show that \eqref{eq:long-basic-curving} 
reduces to \eqref{eq:basic-curving} and \eqref{eq:long-basic-curvature} reduces to \eqref{eq:basic-curvature}.
\end{proof}

\subsection{Other choices of general cup product bundle gerbes}
By comparing the curving and three-curvature of the Weyl bundle gerbe with the curving and three-curvature of the pullback of the basic bundle gerbe from Proposition \ref{curvature}, we can begin to establish a relationship between these bundle gerbes. To do so, we require the following key observation. 
\begin{lemma}
	\label{lemmalogandepsilon} 
	For each $i=1, \dots, n$ and $(z, w, t) \in Y_T^{[2]}$, $$\varepsilon_i(z, w, t) = {\frac{1}{2\pi i}}\left(\log_zp_i(t) - \log_wp_i(t)\right). $$
\end{lemma}
\begin{proof} 
	Recall Definition \ref{defnepsilon}.  Let $(z, w, t)\in Y_T^{[2]}$ with $z>w$. If $w<p_i(t)<z$, $\log_zp_i(t) - \log_wp_i(t) = 2\pi i$. Otherwise, this difference is zero. Therefore in general 
	$$
	\text{log}_zp_i(t) - \text{log}_wp_i(t) = \begin{cases}
	2\pi i &\text{if } z>p_i(t) > w\\
	-2\pi i & \text{if } w>p_i(t) > z\\
	0 &\text{otherwise}.
	\end{cases}
	$$ 
	Dividing through by $2\pi i$, we see that this is precisely the definition of $\varepsilon_i$. \end{proof}

It follows from Propositions  \ref{curvaturetwoform} and \ref{curvaturetheorem2} and equations \eqref{eq:basic-curving} and  \eqref{eq:basic-curvature}  that if we let $\varphi_i = -\frac{1}{2 \pi i} \log_z(p_i) \colon Y_T \to \RR$, then $\delta (\varphi_i) = 
\varepsilon_i$ and we can construct a general cup product curving $f$ and curvature $\omega$ for the pullback of the basic
bundle gerbe which would be 

\begin{align*}
f &=  \sum_{k=1}^n \left(\frac{-1}{2 \pi i}\log_z p_k \right) \tr(P_k dP_k dP_k) = f_{p^*b} -  (\pi_{p^*b})^*\beta \\
\omega &= -\frac{1}{2\pi i} \sum_{k=1}^n  p_k^{-1}dp_k  \textup{tr}(P_kdP_kdP_k) = \omega_{p^*b} -  d \beta .
\end{align*}

We can ask more generally if there is a 
choice of functions $f_i: Y_T^{[2]} \to \Z$ and $\varphi_i \colon Y_T \to \RR$
satisfying $\delta(\varphi_i) = f_i$ 
such that the curving $f$ and three-curvature $\omega$ of the resulting
general cup product bundle gerbe of $J_i$ and $f_i$ satisfy $f_{p^*b} = f$ and $\omega_{p^*b} = \omega$.
For this to hold we would require functions  $\alpha_i:T\times \su\to \R$ for $i= 1, \dots, n$ 
  such that $\beta= \sum_{i=1}^n \alpha_i \textup{tr}(P_idP_idP_i)$. 
We claim that for $n > 2$ this is not possible.  

\begin{proposition}
	If $n>2$, there do not exist functions 
  $\alpha_i:T\times \su\to \R$ for $i= 1, \dots, n$ 
  such that $\beta= \sum_{i=1}^n \alpha_i \textup{tr}(P_idP_idP_i)$. 
\end{proposition}
\begin{proof}
	By Lemma~\ref{lemmaappendixA3}, there 
  exists $\beta_{ij}$ such that $\beta$ 
  decomposes into the sums 
  \[
  \beta = \sum_{i<j<n} (\beta_{ij}-\beta_{in}+\beta_{jn})
  \textup{tr}(P_jdP_idP_i) - \sum_{i<n} \beta_{in}\textup{tr}(P_idP_idP_i).
  \] 
  Moreover, the first of these summations is non-zero by 
  Lemma~\ref{lemmaappendixA3} (3). By 
  Lemma~\ref{projections} (4), $\sum_{k=1}^n \alpha_k \text{tr}(P_kdP_kdP_k) 
  = \sum_{k<n} (\alpha_k-\alpha_n)\text{tr}(P_kdP_kdP_k)$. 
  Therefore it suffices to show that 
  \begin{align}\label{snow} 
  \text{span}\left\{\text{tr}(P_jdP_idP_i) \ | \ i< j < n \right\} 
  \cap \text{span}\left\{\text{tr}(P_kdP_kdP_k) \ | \ k<n\right\} = \{ 0 \}.
  \end{align} 
  Let $E_{ij}$ be the $n\times n$ matrix with a $1$ in the $(i, j)$ entry and zeros elsewhere. Set $O_i:= E_{ii}$. Then 
  \begin{align}
	E_{ij}E_{kl} &= \delta_{jk}E_{il} \label{1}\\
	O_iE_{kl} &= \delta_{ik}E_{kl} \label{2} \\
	E_{kl}O_i &= \delta_{il}E_{kl}\label{3}.
	\end{align}
	The  root spaces for the Lie algebra $LSU(n)$ are spanned by matrices of the 
  form   $A_{ij}^{\mu} = \mu E_{ij} - \overline{\mu}E_{ji} $ 
  for $\mu\in \C$. Let
  $\gamma(t) = g\text{exp}(tX)T$ be a curve in $\su$ {through $gT$}. So 
  \[
  P_i(\gamma(t)) = g\exp(tX)O_i \exp(-tX)g^{-1}
  \] 
  and $dP_i(gX) = g[X, O_i]g^{-1}.$ Using this, 
  it can be verified easily that 
  \begin{align*}\text{tr}(P_jdP_idP_i)(gX, gY) &= -\text{tr}(O_jXO_iY)+\text{tr}(O_jYO_iX)  \\ 
  \text{tr}(P_idP_idP_i)(gX, gY) &= \text{tr}(-O_iXY) + \text{tr}(O_iXO_jY) \\ 
  & \ \ + \text{tr}(O_jYX)-\text{tr}(O_iYO_jX). 
  \end{align*} 
  In particular, using equations~(\ref{1}) -- (\ref{3}), 
  a simple computation yields 
  \begin{align}\text{tr}(P_jdP_idP_i)(gA_{in}^{\mu}, gA_{in}^{\lambda}) &= 0\label{4} \\ 
  \text{and  } \text{tr}(P_idP_idP_i)(gA_{in}^{\mu}, gA_{in}^{\lambda}) &= \delta_{ki}(\overline{\lambda}\mu - \mu \overline{\lambda}).\label{5} 
  \end{align} 
	Consider an element 
  \[
  \sum_{i<j<n} b_{ij}\text{tr}(P_jdP_idP_i) = \sum_{k<n}\alpha_k \text{tr}(P_kdP_kdP_k)
  \] 
  in the intersection from~\eqref{snow}. By equations~\eqref{4} and~\eqref{5}, 
  evaluating this element at $(gA_{kn}^{\mu}, gA_{kn}^{\lambda})$ 
  yields $0 = \alpha_k (\lambda \overline{\mu} - \mu\overline{\lambda})$. 
  Choosing $\lambda$ and $\mu$ so that $\alpha_k = 0$ for all $k$ proves~\eqref{snow}.
\end{proof}

By the earlier discussion, the following corollary is immediate. 
\begin{corollary}
Let $n > 2$. There does not exist a
choice of functions $f_i: Y_T^{[2]} \to \Z$ and $\varphi_i \colon Y_T \to \RR$
satisfying $\delta(\varphi_i) = f_i$ 
such that the curving $f$ and three-curvature $\omega$ of the resulting
general cup product bundle gerbe of $J_i$ and $f_i$ satisfy $f_{p^*b} = f$ and $\omega_{p^*b} = \omega$.
\end{corollary}

%
%

\section{The stable isomorphism}
\label{sec:stable}

\subsection{Set up of the problem}
Our central aim in this section is to prove that the 
pullback of the basic bundle gerbe by the Weyl map is 
$SU(n)$-stably isomorphic to the Weyl bundle gerbe, i.e.\  
\begin{align} \label{keyequation} \pullbackbasicbg \cong_{SU(n)\text{-stab}} \cupbg. 
\end{align} 
By Definition~\ref{definitionofthefinalcupproductbundlegerbe} and 
Proposition~\ref{pullbackstableisomainresult}, 
\eqref{keyequation} is equivalent to 
\begin{align*} 
 \sideset{}{_\mathrm{R}}\bigotimes_{i=1}^n \left( J_i^{\varepsilon_i}, Y_T\sstimes \su \right)
 \cong_{SU(n)\textup{-stab}} 
\sideset{}{_\mathrm{R}}\bigotimes_{i=1}^n \left( J_i^{d_i}, X_T \stimes \, {SU(n)}/{T}\right) 
\end{align*} 
Since both of these bundle gerbes are general cup product bundle gerbes, Corollary~\ref{cormaincupproductresult} applies to give us the following result.
\begin{proposition}\label{mainresult}
	The pullback of the basic bundle gerbe is $SU(n)$-stably isomorphic to the Weyl bundle gerbe if, for all $i = 1, \dots, n,$ there exist smooth functions $h_i: \left(X_T \times_T Y_T\right) \times \su \to \Z$ such that 
	\begin{align}
	\label{equationofh}
\varepsilon_i(z, w, t) - (x_i - y_i)  = h_i(y, w, t, gT) - h_i(x, z, t, gT)
	 \end{align} 
	 for all $(x= (x_1, \dots, x_n), y= (y_1, \dots, y_n), z, w, t, gT)\in \left(X_T \times_T Y_T\right)^{[2]} \times \su$.
\end{proposition}

\subsection{Finding the stable isomorphism}
It follows from a standard fact in bundle gerbe theory that, if 
equation~\eqref{keyequation} holds with respect to the 
connective data $(\nabla_{p^*b}, f_{p^*b}), (\nabla_c, f_c)$ from 
Propositions~\ref{curvature} and~\ref{corollary data on cup prod}, there exists  a 
trivialising line bundle $R$ with connection $\nabla_R$ and 
$\beta \in \Omega^2\left(T\times \su\right)$ such that 
\begin{align}
 f_{p^*b} - f_c &= F_{\nabla_R} + \pi^*\beta \label{curvingscompare} \\ 
 \omega_{p^*b} - \omega_c &= d\beta    \notag
 \end{align} 
 for $\pi: \left(X_T\times_T Y_T\right) \times \su \to T\times \su$ projection.  
 As in Corollary~\ref{cormaincupproductresult} we take $R$ to be the line bundle 
 \begin{align}
  R:= \Motimes_{i=1}^n J_i^{h_i}\to (X_T\times_T Y_T) \times \su  \notag
 \end{align}
 where we implicitly pull $J_i$ back from $\su$ to $ (X_T \times_T Y_T) \times \su$. {Here, the functions $h_i: (X_T\times_TY_T)\times \su$ are parameters that we aim to define}.
 In this situation we can take $\nabla_R$ to be the product connection and thus 
 \begin{equation}
 \label{eq:R-curvature}
 F_{\nabla_R} = \sum_{i=1}^n h_i \tr(P_i dP_i dP_i ) .
 \end{equation}

First we compare the curving and curvature for the two bundle gerbes.

\begin{proposition}  The pulled back connective data 
for $\left(P_{p^*b}, Y_T \times \su\right)$ from 
\textup{Proposition~\ref{curvature}} and the 
Weyl bundle gerbe connective data for 
$\left(P_c, X_T \times \su\right)$  from \textup{Proposition~\ref{curvaturetheorem1}} satisfy
\begin{align*}
f_{p^*b} - f_c &= \sum_{k=1}^n \left(-\frac{1}{2 \pi i}\log_z p_k + x_i\right) \tr(P_k dP_k dP_k)   +\pi^*\beta \\
\omega_{p^*b} - \omega_c &= -\frac{1}{2\pi i} \sum_{k=1}^n  p_k^{-1}dp_k  \tr(P_kdP_kdP_k) + d \beta.
\end{align*}
\end{proposition}
\noindent It follows by comparison with  
equations~\eqref{curvingscompare} and~\eqref{eq:R-curvature} that we want to take 
$$
h_i(x, z, t, gT) = x_i - \frac{1}{2\pi i}\log_z p_i(t)
$$
for all $i = 1, \dots, n$. It remains to be shown that these $h_i$  
satisfy equation~\eqref{equationofh} and hence define the required stable isomorphism.

\begin{proposition}\label{hisright}
For $i = 1, \dots, n$ define $h_i: \left(X_T \times_T Y_T\right) \times \su \to \Z$ by $$h_i(x, z, t, gT) = x_i - \frac{1}{2\pi i}\log_z p_i(t)$$ for $(x= (x_1, \dots, x_n), z, w, t, gT)\in\left(X_T \times_T Y_T\right) \times \su$. Then $h_i$ is smooth and $$\varepsilon_i(z, w, t) -x_i + y_i= h_i(y, w, t, gT) - h_i(x, z, t, gT)$$ for all $(x= (x_1, \dots, x_n), y= (y_1, \dots, y_n), z, w, t, gT)\in \left(X_T \times_T Y_T\right)^{[2]} \times \su $.
\end{proposition}
\begin{proof}
First, note that $h_i(x, z, t, gT) \in \Z$ since $e^{2\pi i x_i} = p_i(t)$, so upon exponentiating $h_i$ we obtain $e^{2\pi i h_i} = e^{2\pi i x_i}p_i(t)^{-1}= 1$. Smoothness of $h_i$ follows by noting that $\log$ is smooth over the given range as $z\neq p_i(t).$ By Lemma \ref{lemmalogandepsilon}, \begin{align*}
 h_i(y, w, t, gT) - h_i(x, z, t, gT) &= y_i-x_i + \frac{1}{2\pi i}\left(\log_{z}p_i(t) - \log_{w}p_i(t)\right)\\
&= y_i - x_i + \varepsilon_i(z, w, t),
\end{align*} so these are the desired functions $h_i$.
\end{proof}
The next result then follows immediately from Propositions \ref{mainresult} and \ref{hisright}. A more precise statement of this result will be provided in Theorem \ref{finalresult}.
\begin{proposition}\label{prp}
	The Weyl bundle gerbe is $SU(n)$-stably isomorphic to the pullback of the basic bundle gerbe by the Weyl map, i.e. $$\cupbg \cong_{SU(n)\textup{-stab}} \pullbackbasicbg.$$
\end{proposition}

\subsection{Comparing holonomies} Recall that bundle gerbes are \textit{$D$-stably isomorphic} if they are stably isomorphic as bundle gerbes \textit{with connective data}. It is a standard fact that, if two bundles gerbes over a surface are $D$-stably isomorphic, then they have the same holonomy. Therefore if we can show our bundle gerbes have different holonomies on a surface $\Sigma \subset T\times \su$, then the restriction of our bundle gerbes to $\Sigma$ (and hence the original bundle gerbes) cannot be $D$-stably isomorphic, and their $D$-stable isomorphism classes (Deligne classes) will not be equal.

 By our choice of trivialising line bundle, the curvings of the pullback of the basic bundle gerbe and Weyl bundle gerbe satisfy \begin{align} f_{p^*b} &= f_c + F_{\nabla_R} + \pi^*\beta_n \label{wer} \\ \text{for   } \  \beta_n &=  -\frac{i}{4\pi}\sum_{\substack{i, k=1 \\ i\neq k}}^n  p_ip_k^{-1}\text{tr}(P_idP_kdP_k).\label{werrr} \end{align} Here we introduce the notation $\beta_n$ to emphasise that $\beta_n$ is defined on $T\times SU(n)/T$.
It follows from Proposition \ref{prp}, equation (\ref{wer}) and standard facts in holonomy that, for $\Sigma \subset T\times \su$ a surface, the holonomies of the pullback of the basic bundle gerbe and Weyl bundle gerbe satisfy
 \begin{align}\label{aadff} \text{hol}((\nabla_{p^*b}, f_{p^*b}), \Sigma) 
 &= \text{exp}\left(\int_{\Sigma} \beta_n\right) \text{hol}((\nabla_c, f_c), \Sigma).\end{align}
 It could be the case that $\int_\Sigma \beta_n = k2\pi i$ for some $k\in \Z$, implying these holonomies are equal. We next show that there exists a surface $\Sigma_2 \subset T\times {SU(2)}/{T}$ for which $\int_{\Sigma_2} \beta_2 \neq k2\pi i$ for any $k\in \Z$. We will then generalise this result to obtain a surface $\Sigma_n \subset T\times {SU(n)}/{T}$  for which $\text{hol}((\nabla_{p^*b}, f_{p^*b}), \Sigma_n)  \neq \text{hol}((\nabla_c, f_c), \Sigma_n)$.

\begin{proposition}\label{holonomycomputation}
Define a surface $\Sigma_2\subset T\times 
SU(2)/T \cong S^1\times S^2$ by $\Sigma_2 := 
\{e^{{\pi i}/{4}}\}\times S^2$. Then the holonomies 
of the pullback of the basic bundle gerbe over $SU(2)$ 
and the Weyl bundle gerbe over $T\times SU(2)/T$ are not equal over $\Sigma_2$.
\end{proposition}
\begin{proof}
By equation (\ref{aadff}), we need only show 
that $\int_{\Sigma_2} \beta_2 \neq k2\pi i$ for any $k\in \Z$. 
 Since $P_1+P_2=1$ and $p_2 = p_1^{-1}$, by 
 setting $P:= P_1$ and $p := p_1$ in equation~\eqref{werrr} 
 we obtain 
 \[
 \beta_2 = \frac{i}{4\pi}(p^2-p^{-2})\text{tr}(PdPdP). 
 \] 
 It is a standard fact that $\text{tr}(PdPdP)$ is 
 the curvature of the tautological line bundle over 
 $S^2$, which has chern class minus one, i.e. 
 $\frac{i}{2\pi}\int_{S^2} \text{tr}(PdPdP) = -1.$ 
 Therefore 
 \begin{align*}
	\int_{\Sigma_2} \beta_2 &= \frac{ie^{\tfrac{i\pi}{2}} - ie^{-\tfrac{i\pi}{2} }}{4\pi} \int_{S^2} \text{tr}(PdPdP) \\
	&= 	 \frac{-e^{\tfrac{i \pi}{2} } + e^{-\tfrac{i\pi}{2}}}{2\pi}\\
	&= \frac{1}{\pi i} \neq k2\pi i \ \forall \ k \in \Z,
\end{align*}
hence $\text{exp}\left(\int_{\Sigma_2} \beta_2\right)\neq 1$ and the holonomies are not equal over this surface.
\end{proof}
\begin{corollary}
There exists a surface $\Sigma_n \subset T\times \su$ such that \begin{align*}\textup{hol}((\nabla_{p^*b}, f_{p^*b}), \Sigma_n)  \neq \textup{hol}((\nabla_c, f_c), \Sigma_n).\end{align*}
\end{corollary}
\begin{proof}
First, note that surface $\Sigma_2 = \{e^{{\pi i}/{4}}\}\times S^2$ from Proposition \ref{holonomycomputation} is an embedded submanifold of $T\times \su$ with respect to the inclusion $\iota: {SU(2)}/{T_1} \hookrightarrow {SU(n)}/{T_{n-1}}$ defined by $$XT_1 \mapsto \left[
\begin{array}{c|c}
X & 0 \\
\hline
0 & \textbf{I}_{n-2}
\end{array}
\right]T_{n-1}.$$ Here, $T_1, T_{n-1}$ denote the subgroups of diagonal matrices in $SU(2)$ and $SU(n)$ respectively, and $\textbf{I}_{n-2}$ is the $(n-2)\times (n-2)$ identity matrix. Let $\Sigma_n := \iota\left(\Sigma\right)$. By equation (\ref{aadff}) it suffices to show that $$\int_{\Sigma_n} \beta_n = \int_{\Sigma_2} \iota^*\beta_n \neq k2\pi i$$ for any $k\in \Z$. To do so, we prove that $\iota^*\beta_n = \beta_2$, hence $\int_{\Sigma_n} \beta_n \neq k2\pi i$ by the proof of Proposition \ref{holonomycomputation}. We compute $\iota^*\beta_n$ as follows. Recall that the maps $p_i:T\to S^1$ were defined as projection onto the $i$-th diagonal. Clearly $$p_i\circ \iota = \begin{cases}
p_i & \text{if} \ i = 1, 2\\
1 & \text{if} \ 2< i \leq n.
\end{cases} $$
Further recall that $P_i$ was defined to be orthogonal projection onto $J_i := \C\times_{p_i}SU(n)$, where $p_i$ was the relation $(z, s)\sim_{p_i}(p_i(t^{-1})z, st)$ for all $(z, s)\in \C\times SU(n)$. Now, when the maps $p_i$ are the constant value $1$, this relation is equality, and $J_i\to \su$ is isomorphic to the trivial line bundle over $\su$. In this case, $P_i$ will be the constant projection onto the span of $e_i$, the $i$-th standard basis vector of $\C^n$. That is, $P_i = O_i$ for $O_i$ the matrix with a $1$ in the $(i, i)$ position and zeros elsewhere. Therefore $$P_i\circ \iota = \begin{cases}
P_i & \text{if} \ i=1, 2\\
O_i & \text{if} \ 2< i \leq n.
\end{cases} $$ Of course, $dO_i = 0$, so any term of the form $\text{tr}(P_k dP_i dP_i)$ for $i>2$ in our expression for $\beta_n$ in (\ref{werrr}) will equal zero. Furthermore, any term of the form $\text{tr}(P_i dP_k dP_k)$ for $i>2$ will also be zero, by Lemma \ref{projections} (2). So $\iota^*\beta_n = \beta_2$ as required. \end{proof} The following corollary is immediate from our earlier discussion. 
\begin{corollary}
	There does not exist a $D$-stable isomorphism of   $\left(P_c, X\right)$ and $p^{-1}(P_b, Y)$ with respect to the connective data $(\nabla_c, f_c)$ and $(\nabla_{p^*b}, f_{p^*b})$.
\end{corollary}

The results of Sections \ref{sec:basic-pullback} and \ref{sec:stable} culminate in the following theorem.
\begin{theorem}\label{finalresult} Let $p^{-1}(P_b, Y)$ be the pullback of the basic bundle gerbe \textup{(Definition \ref{bbbbbbg})} by the Weyl map with connective data $(\nabla_{p^*b}, f_{p^*b})$ and three-curvature $\omega_{p^*b}$ \textup{(Proposition \ref{gbgbgb})}. Let $\left(P_c, X\right)$ be the Weyl bundle gerbe \textup{(Definition \ref{definitionofthefinalcupproductbundlegerbe})} with connective data $(\nabla_c, f_c)$ and three-curvature $\omega_c$ \textup{(Propositions \ref{jello} - \ref{corollary data on cup prod})}. Then \begin{enumerate}[(1),font=\upshape]
		\item 	there  is an $SU(n)$-equivariant stable isomorphism over $T\times \su$ $$\cupbg \cong_{SU(n)\textup{-stab}} \pullbackbasicbg,$$ with trivialising line bundle $$ R:= \Motimes_{i=1}^n \pi_{2}^{-1}(J_i)^{h_i}\to (X_T\times_T Y_T) \times \su$$ for $\pi_{2}: (X_T \times_T Y_T) \times \su \to  \su$ projection and \begin{gather*} h_i: \left(X_T \times_T Y_T\right) \times \su \to \Z \\ ((x_1, \dots, x_n), z, t, gT) \mapsto x_i - \frac{1}{2\pi i}\log_z p_i(t);\end{gather*} 
		\item if $\nabla_R$ is the connection on $R$ induced by $\nabla_{J_i}$ \textup{(Proposition \ref{curvaturetwoform})}, then \begin{align*}  f_{p^*b} - f_c &= F_{\nabla_R} + \pi^*\beta\\ \text{and } \ \omega_{p^*b}-\omega_c &= d\beta  \end{align*} for $\pi: \left(X_T\times_T Y_T\right) \times \su \to T\times \su$ projection and $$\beta =  -\frac{i}{4\pi}\sum_{\substack{i, k=1 \\ i\neq k}}^n  p_ip_k^{-1}\textup{tr}(P_idP_kdP_k)$$ where $P_i:T\times \su\times \C^n\to J_i$  is orthogonal projection\textup{;} 
		\item there does not exist a general cup product bundle gerbe of $J_i$ and some functions $f_i:X_T^{[2]}\to \Z$ and 
		 $\varphi_i \colon X_T \to \RR$ with  $\delta(\varphi_i) = f_i$ whose induced connective data \textup{(}following \textup{Proposition \ref{corollary data on cup prod})} has associated three-curvature $\omega = \omega_{p^*b};$
	\item 	there does not exist a $D$-stable isomorphism of   $\left(P_c, X\right)$ and $p^{-1}(P_b, Y)$ with respect to the connective data $(\nabla_c, f_c)$ and $(\nabla_{p^*b}, f_{p^*b})$.
		\end{enumerate}
\end{theorem}

%
%

\appendix

\section{Computational lemmas}

 Here, we present the lemmas used to prove various results in Section \ref{sec:stable}.

\begin{lemma}\label{projections} 
	Let $i, j, k = 1, \dots, n$. Then \begin{enumerate}[(1),font=\upshape]
		\item for distinct $i, j, k$, $\textup{tr}(P_idP_jdP_k) = 0;$
		\item if $i\neq j$, $\textup{tr}(P_idP_jdP_j) = -\textup{tr}(P_jdP_idP_i);$
		\item $\sum_{k=1}^n \tr(P_k dP_k dP_k) = 0;$
		\item $\sum_{i=1}^n \alpha_i \textup{tr}(P_idP_idP_i) = \sum_{i=1}^{n-1} (\alpha_i - \alpha_n)\textup{tr}(P_idP_idP_i).  $
	\end{enumerate}
\end{lemma}
\begin{proof}
	To prove (1), note that $P_iP_j = 0$ if $i\neq j$, and $dP_i = dP_iP_i + P_i dP_i$ (where we obtain the second equation by differentiating $P_i^2 = P_i$). So for distinct $i, j$ and $k$ we have \begin{align*}
	\text{tr}(P_i dP_jdP_k) &= \text{tr}(P_i (P_jdP_j + dP_jP_j) dP_k)\\
	&= \text{tr}(P_idP_jP_jdP_k)\\
	&= \text{tr}(P_idP_jP_j(P_kdP_k + dP_kP_k))\\
	&= \text{tr}(P_idP_jP_jdP_kP_k)\\
	&= \text{tr}(P_kP_idP_jP_jdP_k)=0,\end{align*} thereby proving $(1)$. Next, by differentiating the identity $P_iP_j = 0$, we obtain $dP_iP_j = -P_idP_j$ for $i\neq j$. Therefore, using $(1)$ and that $\sum_{i=1}^n dP_i = 0$, we obtain \begin{align*}
	\text{tr}(P_idP_jdP_j) &= -\text{tr}(dP_iP_jdP_j) \\
	&= \text{tr}(P_jdP_jdP_i)\\
	&= \text{tr}\left(P_j \left(-\textstyle \sum_{k\neq j} dP_k\right)dP_i \right)\\
	&= -\sum_{k\neq j} \text{tr}(P_jdP_kdP_i)\\
	&= -\text{tr}(P_jdP_idP_i),
	\end{align*} thereby proving $(2)$. 
	For (3) we use (2). We have 
\begin{align*}
\sum_{k=1}^n \tr(P_k dP_k dP_k) & = - \sum_{i \neq k } \tr(P_i dP_k dP_k) \\
& = - \sum_{i < k} \tr(P_i dP_k dP_k) -  \sum_{k < i } \tr(P_i dP_k dP_k) \\
& = - \sum_{i < k} \tr(P_i dP_k dP_k) -  \sum_{i < k } \tr(P_k dP_i dP_i) \\	
& = - \sum_{i < k} \tr(P_i dP_k dP_k) +  \sum_{i < k } \tr(P_i dP_k dP_k) \\	
& = 0.
\end{align*}	Lastly, by (2), $\textup{tr}(P_kdP_ldP_l) + \textup{tr}(P_ldP_ldP_k)  = 0$. 
	 Using this, together with $(1)$ we obtain \begin{align*} \sum_{i=1}^n \alpha_i\textup{tr}(P_idP_idP_i) &= \sum_{i=1}^{n-1}\alpha_i\textup{tr}(P_idP_idP_i) -\alpha_n \textup{tr}\left( \left(\sum_{m=1}^{n-1} P_m\right)  \left(\sum_{k=1}^{n-1} dP_k\right)\left(\sum_{l=1}^{n-1} dP_l\right)\right) \\ 
	&= \sum_{i=1}^{n-1}\alpha_i\textup{tr}(P_idP_idP_i) -\alpha_n \sum_{i=1}^{n-1}\textup{tr}(P_idP_idP_i) \\ &\ \ \ - \alpha_n \sum_{\substack{k, l=1 \\ k\neq l}}^{n-1} \textup{tr}(P_kdP_ldP_l) + \textup{tr}(P_ldP_ldP_k)      \\ 
	&= \sum_{i=1}^{n-1}(\alpha_i - \alpha_n)\textup{tr}(P_idP_idP_i),
	\end{align*} 
	proving (4).
\end{proof}

\begin{lemma}
\label{lemmaappendixA3}
Consider $\beta =   -\frac{i}{4\pi}\sum_{\substack{i, k=1 \\ i\neq k}}^n  p_ip_k^{-1}\textup{tr}(P_idP_kdP_k)$. Then \begin{enumerate}[(1),font=\upshape] \item there exist coefficients $\beta_{ij}$ such that $$\beta = \sum_{i<j\leq n} \beta_{ij}\textup{tr}(P_jdP_idP_i);$$
	\item if $n>2$, these $\beta_{ij}$ satisfy $$\beta = \sum_{i<j<n} (\beta_{ij}-\beta_{in}+\beta_{jn})\textup{tr}(P_jdP_idP_i) - \sum_{i<n} \beta_{in}\textup{tr}(P_idP_idP_i); $$ 
	\item if $n>2$, these $\beta_{ij}$ satisfy $$\sum_{i<j<n} (\beta_{ij}-\beta_{in}+\beta_{jn})\textup{tr}(P_jdP_idP_i)\neq 0.$$
	\end{enumerate} 
\end{lemma}
\begin{proof}
It follows from Lemma \ref{projections} (2) that $\beta = \sum_{i<j\leq n} (p_jp_i^{-1} - p_ip_j^{-1})\textup{tr}(P_jdP_idP_i),$ so by setting $\beta_{ij}:= p_jp_i^{-1} - p_ip_j^{-1}$ we obtain (1). It follows that we can write \begin{align*}
\beta &= \sum_{i<j<n} \beta_{ij}\textup{tr}(P_jdP_idP_i) + \sum_{i<n} \beta_{in}\textup{tr}(P_ndP_idP_i)\\
	&=  \sum_{i<j<n} \beta_{ij}\textup{tr}(P_jdP_idP_i) - \sum_{i<n}\sum_{j=1}^{n-1} \beta_{in}\textup{tr}\left( P_jdP_idP_i\right) \\
	&= \sum_{i<j<n} \beta_{ij}\textup{tr}(P_jdP_idP_i) - \sum_{i<j<n}\beta_{in}\textup{tr}\left( P_jdP_idP_i\right) \\&- \sum_{j<i<n}\beta_{in}\textup{tr}\left( P_jdP_idP_i\right) - \sum_{i<n}\beta_{in}\textup{tr}\left( P_idP_idP_i\right)\\
	&= \sum_{i<j<n} (\beta_{ij}-\beta_{in}+\beta_{jn})\textup{tr}(P_jdP_idP_i) - \sum_{i<n}\beta_{in}\textup{tr}\left( P_idP_idP_i\right).
\end{align*}

For (3), consider an element \[
\textbf{S} := \begin{bmatrix}
\textbf{T}& 0  \\
0 & \textbf{I}_{n-2}
\end{bmatrix} \in T
\] 
for $\textbf{T}:= \begin{bmatrix}
t& 0  \\
0 & t^{-1}
\end{bmatrix}$, $t\in U(1)$, and $\textbf{I}_{n-2}$ the $(n-2)\times (n-2)$ identity matrix. Clearly $\beta_{ij} = \beta_{in}$ and $\beta_{jn} =0$ if $j>2$ evaluated at $\textbf{S}$. Therefore the only non-zero coefficient in this summation evaluated at \textbf{S} is $\beta_{12} - \beta_{1n}+\beta_{2n} = 2t - 2t^{-1}+t^{-2}-t^{2}$. The result then follows by choosing $t$ such that this coefficient is non-zero. 
\end{proof}

%
%

\end{document}